\documentclass[11pt,a4paper]{article}

\usepackage[T1]{fontenc}
\usepackage[utf8]{inputenc}
\usepackage{csquotes}

\usepackage{amsfonts, amsmath, amssymb, amsthm, dsfont, mathrsfs, latexsym, euscript}

\usepackage[
backend=biber,
style=ieee,
sorting=nyt,
url = true
]{biblatex}
\usepackage{tikz}
\usepackage{xcolor}
\usepackage{caption, subcaption}
\usepackage{hyperref}
\usepackage[left=2cm,right=2cm,top=2cm,bottom=2cm]{geometry}

\def\EE{\mathbb{E}}
\def\FF{\mathbb{F}}
\def\GG{\mathbb{G}}

\def\LL{\mathbb{L}}

\def\NN{\mathbb{N}}

\def\PP{\mathbb{P}}

\def\RR{\mathbb{R}}

\def\Bcal{\mathcal{B}}
\def\Ccal{\mathcal{C}}

\def\Fcal{\mathcal{F}}
\def\Gcal{\mathcal{G}}

\def\Ical{\mathcal{I}}

\def\Lcal{\mathcal{L}}

\def\Ncal{\mathcal{N}}

\def\Pcal{\mathcal{P}}

\def\Scal{\mathcal{S}}
\def\Tcal{\mathcal{T}}

\def\Ntil{\Tilde{N}}

\newcommand{\eps}{\varepsilon}
\def\E#1{\mathbb{E}\left[ #1 \right]}

\def\ind#1{\mathds{1}_{\{#1\}}}
\def\1{\mathds{1}}

\newtheorem{theorem}{Theorem}[section]
\makeatletter
\renewcommand{\thetheorem}{%
  \ifnum\c@subsection=0
    \thesection.\number\c@theorem
  \else
    \thesubsection.\number\c@theorem
  \fi}
\makeatother

\newtheorem{corollary}[theorem]{Corollary}
\newtheorem{lemma}[theorem]{Lemma}
\newtheorem{proposition}[theorem]{Proposition}
\newtheorem{definition}[theorem]{Definition}
\newtheorem{assumption}[theorem]{Assumption}
\newtheorem{notation}[theorem]{Notation}
\newtheorem{remark}[theorem]{Remark}

\numberwithin{equation}{section}
\makeatletter
\renewcommand{\theequation}{%
  \ifnum\c@subsection=0
    \thesection.\number\c@equation
  \else
    \thesubsection.\number\c@equation
  \fi}
\makeatother

\def\z{\mathrsfs{Z}}
\DeclareSymbolFontAlphabet{\mathrsfs}{rsfs}

\title{Normal approximation of Functionals of Point Processes: Application to Hawkes Processes.}
\author{Laure Coutin\footnote{UPS, IMT UMR CNRS 5219, Universit\'e de Toulouse, 135 avenue de Rangueil 31077 Toulouse Cedex 4 France. \; Email: \texttt{laure.coutin@math.univ-toulouse.fr}} \and Benjamin Massat\footnote{UPS, IMT UMR CNRS 5219, Universit\'e de Toulouse, 135 avenue de Rangueil 31077 Toulouse Cedex 4 France. \; Email: \texttt{benjamin.massat@math.univ-toulouse.fr}} \and Anthony Réveillac\footnote{INSA de Toulouse, IMT UMR CNRS 5219, Universit\'e de Toulouse, 135 avenue de Rangueil 31077 Toulouse Cedex 4 France. \; Email: \texttt{anthony.reveillac@insa-toulouse.fr}} }

\date{\today}

\begin{document}
\maketitle

\begin{abstract}
    In this paper, we derive an explicit upper bound for the Wasserstein distance between a functional of point processes and a Gaussian distribution. Using Stein's method in conjunction with Malliavin's calculus and the Poisson imbedding representation, our result applies to a variety of point processes including discrete and continuous Hawkes processes. In particular, we establish an explicit convergence rate for stable continuous non-linear Hawkes processes and for discrete Hawkes processes. Finally, we obtain an upper bound in the context of nearly unstable Hawkes processes.
\end{abstract}


\section{Introduction}
The Nourdin-Peccati methodology initiated in \cite{nourdin_normal_2012} has opened the way to new quantitative limit theorems by combining Stein's method with Malliavin's calculus. Initially defined for Gaussian functionals, this approach has been extended to the normal approximation for Poisson functionals by Peccati, Solé, Taqqu and Utzet in \cite{peccati_steins_2010}, and further investigated in several works in the literature (we refer for instance to the work of Last, Peccati and Schulte in \cite{last_normal_2016} which constitutes one of the last development of this theory with the use of Mehler's formula). A major issue is the quantification of functional Central Limit Theorems; this is adresses in \cite{besancon_diffusive_2024} by Besançon, Coutin and Decreusefond using other techniques. Finally, another line of research (which this paper contributes to) consists in coupling the Nourdin-Peccati methodology with the so-called Poisson Imbedding representation. The latter consists in representing a counting process as a solution of an SDE driven by some Poisson random measure. We refer to Brémaud and Massoulié \cite{bremaud_stability_1996} for a discussion about this representation.
Combining the Poisson Imbedding representation with the Nourdin-Peccati methodology, Torrisi gave the first bound for functionals of counting processes with stochastic intensities in \cite{torrisi_gaussian_2016}. This class of processes includes the non-linear Hawkes process $H$ with the stochastic intensity $\lambda := \left( \lambda_t\right)_{t\geq 0}$ given as
$$\lambda_t = h\left( \mu + \int_{(0,t)} \phi(t-s) dH_s \right), \quad t\geq 0,$$
where $\mu \in \RR$, $h:\RR \to \RR_+$ is $\alpha$-Lipschitz and $\phi:\RR_+\to \RR$ is such that $\alpha \Vert \phi \Vert_1 <1$. However, this bound does not converge to $0$ as $T$ tends to $+\infty$, even when we consider linear Hawkes processes ($h(x)=x$); more precisely, Theorem 3.1  of \cite{torrisi_gaussian_2016} yields
\begin{equation}
\label{eq: maj_torrisi}
    d_W \left( \frac{H_T - \int_0^T \lambda_t dt}{\sqrt{T}}, G\right) =O(1).
\end{equation}
where $G\sim \Ncal\left(0,\sigma^2 \right)$ and $d_W$ denotes the Wasserstein distance. This result can be seen as a one marginal quantification of the functional CLT
$$\left(\frac{H_{tT}- \int_0^{tT} \lambda_s ds}{\sqrt{T}}\right)_{t\in[0,1]} \xrightarrow[]{T\to +\infty} \left( \sigma B_t\right)_{t\in [0,1]}$$
which has been obtained by Bacry, Delattre, Hoffmann and Muzy in \cite{bacry_limit_2013} in the linear case and extended by Zhu in \cite{zhu_nonlinear_2013} in the non-linear case. In a series of papers by Hillairet, Huang, Khabou, Privault and Réveillac (see \cite{hillairet_malliavin-stein_2022}, \cite{khabou_malliavin-stein_2021}, \cite{khabou_normal_2024}), \eqref{eq: maj_torrisi} has become in the linear case,
\begin{equation}
\label{eq: maj_intro}
    d_W \left( \frac{H_T - \int_0^T \lambda_t dt}{\sqrt{T}}, G\right) \leq \frac{C}{\sqrt{T}}.
\end{equation}
Note that a version with multivariate compound Hawkes process has been given in \cite{khabou_malliavin-stein_2021} by Khabou. Moreover, Coutin and Quayle proved a CLT in \cite{quayle_etude_2022} for discrete Hawkes processes. More recently, Deschatre, Gruet and Lotz in \cite{deschatre_limit_2025} generalize results of \cite{bacry_limit_2013} by considering Hawkes process whose baseline and kernel change over time. In particular, they consider a class of locally stationary Hawkes process that has been introduced by Roueff, von Sachs and Sansonnet in \cite{roueff_locally_2016}.

Another significant convergence result for Hawkes processes was obtained by Jaisson and Rosenbaum \cite{jaisson_limit_2015}, who studied nearly unstable Hawkes processes and obtained a non Gaussian limit. Besides, Horst and Xu in \cite{horst_functional_2024} established limit theorems and provided quantitative results for subcritical ($\Vert \phi \Vert_1 <1$) and critical ($\Vert \phi \Vert_1 =1$) linear Hawkes processes.  
In particular, they extended existing results on the asymptotic behavior of subcritical and critical Hawkes processes by considering kernel functions $\phi$ under weaker assumptions than those previously required in the literature. Moreover, they quantified their results, showing that the convergence rate is strongly influenced by another fundamental convergence property involving the unique solution $\Psi$ of the Volterra equation:  
$$ \Psi = \phi + \phi \ast \Psi.$$
Specifically, they demonstrated that the upper bound is governed by the difference between $\Psi(T \cdot)$ and an exponential function.
\\\\
\noindent
In this paper, we consider a general quantity $F^T$ that unifies the aforementioned limit theorems. More precisely, we give in Section \ref{sec: gen estimate} a general bound of $d_W\left( F^T,G\right)$ in Theorem \ref{thm: main} where $F^T$ is defined by
$$F^T := \iint_{[0,T]\times \RR_+} \1_{\theta \leq f^T(t)} \frac{1}{\sqrt{T g^T(t)}} \left( N(dt,d\theta)-dtd\theta\right)$$
in terms of the quantities $f^T$ and $g^T$ that may be random. Here, $N$ stands for a two-parameter Poisson measure. After providing a quantification for the CLT for general counting processes (see Proposition \ref{prop:main2}), we develop our result in four particular directions. We start by providing the same bound as \eqref{eq: maj_intro} when $H$ stands for a non-linear Hawkes process. Our result also allows us to get a similar result to \eqref{eq: maj_intro} in the case of locally stationnary Hawkes proceess and in the case of discrete Hawkes process. Finally, we address the nearly unstable Hawkes process by providing some partial results.\\\\
\noindent
We proceed as follows. The Poisson imbedding and elements of Malliavin's calculus are presented in Section \ref{sec: Notation}. The main results are collected in Section \ref{sec: gen estimate} and applied to (linear and non-linear) Hawkes processes in Section \ref{sec: App Hawkes}, to locally stationnary Hawkes processes in Section \ref{sec:locally_stat} and to discrete Hawkes processes in Section \ref{sec:discrete}. The estimates on the nearly unstable case are collected in Section \ref{sec:nearly}. Finally technical lemmata are postponed in Section \ref{sec:technical}.


\section{Notations and preliminaries}
\label{sec: Notation}
To start with, let us introduce some notations. We denote by $\NN$ (resp. $\NN^*$) the set of non-negative (resp. positive) integers, that is $\NN := 0, \, 1,\, 2,\dots$ (resp. $\NN^* :=  1,\, 2,\dots$ ). We also define respectively $\RR_+ := [0, +\infty)$ and $\RR_+^* := (0, +\infty)$ the real non-negative and positive line. Moreover, for any $(a,b)\in \RR^2$, we may use the notation $[a,b]$ either for $a\leq b$ or for $a\geq b$. We also denote by $\Bcal(E)$ the Borel set of some topological space $E$.\\
Besides, for any random variable $F$, we will write $\Lcal_F$ its distribution.\\

Throughout this paper, $C$ denotes a constant that may change from line to line and which is independent of $T$ which is the parameter to be controlled in our results.\\

\subsection{Elements of stochastic analysis on the Poisson space}

Let the space of configurations
$$ \Omega:=\left\{\omega=\sum_{i=1}^{n} \delta_{(t_{i},\theta_i)} \mid t_0 < t_1 < \cdots < t_n, \; (\theta_i)_{i=1,\dots,n} \in \left(\RR_+\right)^n, \; n\in \NN \cup\{+\infty\} \right\}.$$
Each realization of a counting process is represented as an element $\omega$ in $\Omega$ which is a $\NN$-valued measure on $\RR \times \RR_+$. 

Let $\Fcal^N$ be the $\sigma$-field associated to the vague topology on $\Omega$, and $\PP$ be a probability measure under which the random measure $N$ defined as:
$$ N(B)(\omega):=\omega(B), \quad B \in \Bcal \left( \RR\times \RR_+\right),$$
is a random measure with intensity $1$ (so that $N(B)$ is a Poisson random variable with intensity $\pi(B)$ for any $B\in \Bcal \left( \RR \times \RR_+ \right)$ and where $\pi$ denotes the Lebesgue measure). We set $\FF^N:=(\Fcal_t^N)_{t\in \RR}$ the natural history of $N$, that is $$\Fcal_t^N:=\sigma \left(N( \Tcal  \times B), \; \Tcal \subset \Bcal((-\infty,t]), \; B \in \Bcal(\RR_+) \right).$$ Let also, $\Fcal_\infty^N:=\lim_{t\to+\infty} \Fcal_t^N$. The expectation with respect to $\PP$ is denoted by $\E{\cdot}$. For $t\in \RR$, we denote by $\EE_t[\cdot]$ the conditional expectation $\E{\cdot \vert \Fcal_t^N}$.\\\\
\noindent
We recall some elements of stochastic analysis on the Poisson space, especially the shift operator, the Malliavin derivative and its dual operator: the divergence. 
The following essentially is from Hillairet et al. in \cite{hillairet_expansion_2023}.

\begin{definition}[Shift operator]
\label{definition:shift}
    We define for $(t,\theta)$ in $\RR \times \RR_+$ the measurable map $\eps_{(t,\theta)}^+ : \Omega \to \Omega $ where for any $A$ in $\Bcal(\RR \times \RR_+)$
    $$(\eps_{(t,\theta)}^+(\omega))(A) := \left\lbrace \begin{array}{l} \omega(A \setminus \{(t,\theta)\}) +1, \quad \textrm{if } (t,\theta)\in A,\\\omega(A \setminus \{(t,\theta)\}) , \quad \textrm{else.}\end{array}\right. .$$
\end{definition}

\begin{lemma}
    \label{lemma:mesur}
    Let $t \in \RR$ and $F$ be an $\Fcal_t^N$-measurable random variable. Let $v > t$ and $\theta\geq 0$. It holds that 
    $$ F\circ\eps_{(v,\theta)}^+ = F, \quad \PP-a.s.. $$
\end{lemma}

\begin{definition}[Malliavin derivative]
    For $F$ in $\LL^2(\Omega,\Fcal_\infty^N,\PP)$, we define $D F$ the Malliavin derivative of $F$ as 
    $$ D_{(t,\theta)} F := F\circ \eps_{(t,\theta)}^+ - F, \quad (t,\theta) \in \RR_+^2.  $$
\end{definition}

\noindent The following definition is a by-product of \cite[Theorem 1]{picard_formules_1996} (see also \cite{nualart_anticipative_1990}).

\begin{definition}
    Let $\Ical$  be the sub-sigma field of $\Bcal(\RR \times \RR_+)\otimes \Fcal^N$ of stochastic processes $Z:=(Z_{(t,\theta)})_{(t,\theta) \in \RR \times \RR_+}$ in $L^1(\Omega \times \RR \times \RR_+,\PP\otimes \pi)$ such that 
    $$ D_{(t,\theta)} Z_{(t,\theta)} = 0, \quad \textrm{ for a.s. } (t,\theta) \in \RR \times \RR_+.$$
\end{definition}

\begin{remark}
\label{rem:Nshift}
    Let $(t_0,\theta_0)$ in $\RR \times \RR_+$, $(s,t)$ in $\RR \times \RR_+$ with $t_0 < s < t$.\\
    For $\mathcal T \in \{(s,t), (s,t], [s,t), [s,t]\}$ and $B$ in $\Bcal(\RR_+)$,  we have that : 
    $$ N\circ \eps_{(t_0,\theta_0)}^+ (\Tcal \times B) = N (\Tcal \times B). $$
\end{remark}

\begin{definition}
\label{def: div op}
    We set $\Scal$ the set of stochastic processes $Z:=(Z_{(t,\theta)})_{(t,\theta) \in \RR\times \RR_+}$ in $\Ical$ such that $Z$ is predictable according to the natural history and: 
    $$ \E{\iint_{\RR\times \RR_+} \left|Z_{(t,\theta)}\right|^2 dt d\theta} + \E{\left(\iint_{\RR\times \RR_+} Z_{(t,\theta)} dt d\theta \right)^2}<+\infty .$$
    For $Z$ in $\Scal$, we set the divergence operator with respect to $N$ as  
    \begin{equation}
    \label{eq:delta}
        \delta(Z):=\iint_{\RR\times \RR_+} Z_{(t,\theta)} N(dt,d\theta) - \iint_{\RR\times \RR_+} Z_{(t,\theta)} dt d\theta.
    \end{equation}
\end{definition}

\noindent We conclude this section with the integration by parts formula on the Poisson space (see \cite[Remark 1]{picard_formules_1996}) and the Heisenberg equality.

\begin{proposition}[See \textit{e.g.} \cite{picard_formules_1996}]
\label{prop:IPP}
    Let $F$ be in $\LL^2(\Omega,\Fcal_\infty^N,\PP)$ and $Z=(Z_{(t,\theta)})_{(t,\theta) \in \RR\times \RR_+}$ be in $\Scal$. We have that 
    \begin{equation}
    \label{eq:IBPPoisson}
        \E{F \delta(Z)} = \E{\iint_{\RR\times \RR_+} Z_{(t,\theta)} D_{(t,\theta)} F dt d\theta}.
    \end{equation}
\end{proposition} 

\begin{proposition}[Heisenberg commutation property]
\label{prop: heisenberg}
    For $Z\in \Scal$,
    $$D_{(t,\theta)}\left(\delta \left( Z\right) \right) = Z_{(t,\theta)}+\delta \left( D_{(t,\theta)}\left( Z\right)\right), \quad (t,\theta)\in \RR\times \RR_+.$$
\end{proposition}

\paragraph{Poisson Imbedding}\label{parag: poisson imbedding}
In their seminal work on the stability of nonlinear Hawkes processes, Brémaud and Massoulié \cite{bremaud_stability_1996} introduce a new technique known as "Poisson Imbedding", also called Thinning Algorithm, to represent Hawkes processes using uncompensated Poisson measures. The core idea behind this approach is that any point process can be viewed as being driven by an underlying Poisson measure, facilitating both theoretical analysis and numerical simulation (see \cite{ogata_lewis_1981}). The Poisson Imbedding use a two-parameter Poisson measure which serves as the foundation for generating events of the point process. More precisely, for a point process $H$, its $\Fcal^N$-predictable stochastic intensity $\left(\lambda_t \right)_{t\in\RR}$ can be defined as a function of the history of the process. The point process $H$ is given by
\begin{equation}
\label{eq: point process}
    H\left( C\right) = \iint_{C\times \RR_+} \1_{\theta \leq \lambda_t} N\left( dt,d\theta\right), \quad C\in \Bcal\left(\RR \right).
\end{equation}
Hence, each event of the point process is "triggered" by an event in the underlying Poisson measure when the intensity surpasses a certain threshold. Reciprocally, take $\left(\lambda_t \right)_{t\in\RR}$ a non-negative $\Fcal^N$-predictable process and define the point process $H$ as it is done in \eqref{eq: point process}. Then, $H$ admits the $\Fcal^N$-intensity $\left(\lambda_t \right)_{t\in\RR}$.\\

In the following, we will represents any point processes following the Poisson Imbedding representation.


\subsection{Elements on Stein's method}

Whereas Stein's method has been introduced by Stein in \cite{stein_bound_1972}, the combination of the Malliavin calculus with Stein's method (and known as the Nourdin-Peccati's approach) has been initiated in \cite{nourdin_steins_2009} for the approximation of Gaussian functionals, extended in \cite{peccati_steins_2010} for Poisson functionals (which is closer to our paper). We introduce in this section the original Stein's approach which allows one to derive Inequality (\ref{eq:Stein}) below. We refer to \cite{nourdin_normal_2012} for a complete exposition of the original Stein's method and of the Nourdin-Peccati's approach.

\begin{definition}
\label{def:W}
    Let $F$ and $G$ two random variables defined on some $(\Omega, \Fcal_\infty^N,\PP)$. We define the Wasserstein distance between $\Lcal_F$ and $\Lcal_G$ (or simply between $F$ and $G$) as : 
    $$ d_W(F,G) :=\sup_{u \in \mathrsfs{L}_1} \left| \E{u(F)}-\E{u(G)} \right|,$$
    with $\mathrsfs{L}_1:=\left\{u:\RR \to \RR \textrm{ differentiable a.e. with }\; \|u'\|_\infty \leq 1\right\}$. 
\end{definition}
\noindent
Let $G \sim \Ncal(0,\sigma^2)$. 
We set 
\begin{equation}
\label{eq: def_FW}
    \mathrsfs{F}_W:=\left\{u:\RR \to \RR,\; \textrm{ twice differentiable with } \|u'\|_\infty \leq 1, \; \|u''\|_\infty \leq 2\right\}.
\end{equation}

\noindent
Consider $u$ in $\mathrsfs{L}_1$. Stein proved in \cite{stein_bound_1972}, that there exists a function $v_u$ in $\mathrsfs{F}_W$ solution to the functional equation (named Stein's equation) :
$$ u(x)-\E{u(G)} = \sigma^2 v_u'(x) - x v_u(x), \quad \forall x \in \RR. $$
Plugging $F$ in this equation and taking the expectation, we get that : 
$$ \left|\E{u(F)}-\E{u(G)}\right| = \left|\E{\sigma^2 v_u'(F) - F v_u(F)}\right|.$$
Hence, 
\begin{equation}
\label{eq:Stein}
    d_W(F,G) \leq \sup_{v \in \mathrsfs{F}_W} \left|\E{\sigma^2 v'(F) - F v(F)}\right|.
\end{equation}
In addition, the right hand side is equal to $0$ if and only if $F\sim \Ncal(0,\sigma^2)$.

\section{A general estimate for counting processes}
\label{sec: gen estimate}
In this section, we established an upper bound of the Wasserstein distance between a Gaussian random variable $G$ and another random variable $F^T$ with a specific form detailed in \eqref{def: F^T}.\\

\subsection{Notations and assumptions}
We define the set of positive $\FF^N$-predictable processes as $$\Pcal_+ := \left\{ X: \Omega \times \RR \to \RR^*_+ \mid X \text{ is } \FF^N \text{-predictable}\right\}.$$
Throughout this paper, we will consider families of predictable processes indexed by a parameter $T>0$. More precisely, we set
$$\Ccal := \left\{ y:= \left( y^T_{t\wedge T} \right)_{t>0, T>0} \mid \forall T>0, \, y^T \in \Pcal_+ \right\}.$$
In order to simplify the notation, we will use $\left( y^T_{t} \right)_{t\in [0,T]}$ instead of $\left( y^T_{t\wedge T} \right)_{t>0, T>0}$ throughout the paper.

\begin{notation}
\label{not: f and g}
For $f:= \left(f^T(t)\right)_{t\in[0,T]} \in \Ccal$ and $g:= \left(g^T (t) \right)_{t\in[0,T]} \in \Ccal$, we set $\z := \left(\z^T_{t,\theta} \right)_{(t,\theta)\in [0,T]\times \RR_+}$ and $Z:= \left( Z^T_t\right)_{t\in [0,T]} $ as follows
    $$\z^T_{(t,\theta)} :=  \1_{\theta \leq f^T(t)} \quad \textrm{and}\quad Z_t^T := \frac{\1_{[0,T]}(t)}{\sqrt{T g^T(t)}},\, T>0, \, (t,\theta) \in [0,T]\times \RR_+.$$
    By definition, $Z \in \Ccal$ and for any $\theta \in \RR_+$, $\left(\z^T_{(t,\theta)} \right)_{t\in [0,T]} \in \Ccal .$
\end{notation}

From now, we will make use of Notation \ref{not: f and g} whenever referring to $f$, $g$, $\z$, or $Z$.
To streamline the calculations, it becomes imperative to introduce a set of assumptions, which we shall outline in the following

\begin{assumption}
\label{assump: dans S}
 	$f$ and $g$ in $\Ccal$ are such that for all $T>0$, $ \z^T Z^T \in \Scal$ (where we make use of Notation \ref{not: f and g}). We set $F^T$ as follows:
\begin{equation}
\label{def: F^T}
    F^T := \iint_{[0,T]\times \RR_+} \1_{\theta \leq f^T(t)} \frac{1}{\sqrt{T g^T(t)}} \left( N(dt,d\theta)-dtd\theta\right) = \delta \left( \z^T Z^T \right)
\end{equation}
where $\delta(\cdot)$ is the divergence operator with respect to a Poisson process on $\RR\times \RR_+$ with intensity $dtd\theta$ denoted by $N$.\\
\end{assumption}

\begin{assumption}
\label{assump: initiale}$f$ and $g$ in $\Ccal$ are such that for all $T>0$ and for all $(t,\theta) \in [0,T]\times \RR_+$,
        $$\1_{\theta \geq f^T(t)} f^T \circ \eps^+_{(t,\theta)}= \1_{\theta \geq f^T(t)} f^T \quad \textrm{and} \quad \1_{\theta \geq f^T(t)} 	g^T \circ \eps^+_{(t,\theta)} =\1_{\theta \geq f^T(t)} g^T.$$
        Note that this assumption means that both $f^T$ and $g^T$ are affected by $\eps_{(t,\theta)}^{+}$ only if $(t,\theta)$ is under the graph of $f^T$.
\end{assumption}

Before presenting our theorem, we describe direct consequences of Assumption \ref{assump: dans S}.
\begin{remark}
Note that the numerator in the expression of $F^T$ represents the number of points in $[0,T]$ of a point process on $\RR_+$ with a (predictable stochastic) intensity $f^T$. Moreover, the function $g^T$ acts as a normalization process, regulating the impact of accepted points in the point process with intensity $f^T$.
\end{remark}
\begin{remark}
For $T>0$, $\z^T Z^T \in \Scal$ is equivalent to stating that $f^T$ and $g^T$ satisfy the following conditions:
\begin{equation}
\label{rem: hypo1}
    \E{\int_{0}^T \frac{f^T(t)}{g^T(t)} dt}< +\infty \quad \textrm{and} \quad \E{\left( \int_0^T \frac{f^T(t)}{\sqrt{g^T(t)}} dt\right)^2} < +\infty.
\end{equation}
Most of the time, we will opt to establish \eqref{rem: hypo1} rather than relying solely on the explicit definition of $\Scal$.
\end{remark}

\begin{remark}
\label{rem: explosion_hyp_1}
In this remark, we detailed a consequence of Assumption \ref{assump: dans S} combined with two other natural assumptions. Specifically, assume that $f\in \Ccal$ and $g\in \Ccal$ are such that:
\begin{enumerate}
    \item for all $T>0$, $\z^T Z^T \in S$;
    \item there exists $\mu >0$ such that for all $T>0$, $\min_{t\in [0,T]} g^T(t) \geq \mu$;
    \item there exists $\nu >0$ such that for all $T>0$, $\E{\int_{0}^T \frac{f^T(t)}{g^T(t)} dt}\geq \nu T$.
\end{enumerate}
Then,
$$\E{\left(\int_{0}^T \frac{f^T(t)}{\sqrt{g^T(t)}} dt\right)^2} \geq \nu^2 \mu T^2, \quad T>0.$$

\begin{proof}[Proof of Remark \ref{rem: explosion_hyp_1}]
Let $T>0$. Then, we have that
    $$\nu T \leq \E{ \int_{0}^T \frac{f^T(t)}{g^T(t)} dt} \leq \frac{1}{\sqrt{\mu}}\E{\int_{0}^T \frac{f^T(t)}{\sqrt{g^T(t)}} dt}.$$
We now use Jensen's inequality to get
    $$ \nu T \leq \frac{1}{\sqrt{\mu}}\sqrt{\E{ \left(\int_{0}^T \frac{f^T(t)}{\sqrt{g^T(t)}} dt\right)^2}}. $$
Therefore, we have
$$\E{ \left(\int_{0}^T \frac{f^T(t)}{\sqrt{g^T(t)}} dt\right)^2} \geq  \nu^2 \mu T^2.$$

\end{proof}
\end{remark}

\subsection{Upper bounds for general point processes}
Firstly, we establish in Theorem \ref{thm: main} a comprehensive bound on the Wasserstein distance between a centered Gaussian distribution and our variable $F^T$. We then proceed to a refined instance of the theorem in our Theorem \ref{thm : Maj_g_determinist}, under more stringent assumptions, yielding a bound intricately linked to the variance of $F^T$. Finally, we refine Theorem \ref{thm : Maj_g_determinist} further by substituting the normalization factor of $\sqrt{T}$ with $\Delta_T$, subject to precise conditions. This latter result mentioned yields Proposition \ref{prop:main2}.

\begin{theorem}
\label{thm: main}
    Let $G \sim \Ncal (0,\sigma^2)$ for some $\sigma^2>0$. Let $f$ and $g$ such that Assumptions  $\ref{assump: dans S}$ and $\ref{assump: initiale}$ hold. Then, for any $T>0$,
    \begin{enumerate}
        \item If for any $(t,\theta)\in \RR \times \RR_+$, $\z^T D_{(t,\theta)}\left( Z^T\right) \in \Scal$, then
        \begin{align*}
            d_W(F^T,G) &\leq \E{  \left|\sigma^2- \frac{1}{T} \int_{0}^T \frac{f^T(t)}{g^T(t)} dt\right|} +\frac{1}{T^{3/2}} \E{\int_{0}^T \frac{f^T(t)}{(g^T(t))^{3/2}}  dt}\\
            & \quad +\E{ \iint_{\RR_+^2} \z_{(t,\theta)}^{T} Z_t^T \left|\delta \left( \z^T D_{(t,\theta)}\left(Z^T \right)\right) \right| dtd\theta} \\
            & \quad +\E{ \iint_{\RR_+^2} \z_{(t,\theta)}^{T} Z_t^T \left|\delta \left( \z^T D_{(t,\theta)}\left(Z^T \right)\right) \right|^2 dtd\theta}\\
            & \quad + 2\E{ \iint_{\RR_+^2} \z_{(t,\theta)}^{T} \left|Z_t^T \right|^2 \left|\delta \left( \z^T D_{(t,\theta)}\left(Z^T \right)\right) \right|dtd\theta}.
        \end{align*}
        \item Moreover, if $\z^T D_{(t,\theta)}\left( Z^T\right)= \z^T D_{(t,0)}\left( Z^T\right)$, then
        \begin{align*}
            d_W(F^T,G) &\leq \E{  \left|\sigma^2- \frac{1}{T} \int_{0}^T \frac{f^T(t)}{g^T(t)} dt\right|} +\frac{1}{T^{3/2}} \E{\int_{0}^T \frac{f^T(t)}{(g^T(t))^{3/2}}  dt}\\
            & \quad + \frac{1}{\sqrt{T}}\E{ \int_{0}^T \frac{f^T(t)}{\sqrt{g^T(t)}}\left|\delta \left( \z^T D_{(t,0)}\left(Z^T \right)\right) \right| dt} \\
            & \quad +\frac{1}{\sqrt{T}}\E{ \int_{0}^T \frac{f^T(t)}{\sqrt{g^T(t)}} \left|\delta \left( \z^T D_{(t,0)}\left(Z^T \right)\right) \right|^2 dt}\\
            & \quad + \frac{2}{T}\E{ \int_{0}^T \frac{f^T(t)}{g^T(t)} \left|\delta \left( \z^T D_{(t,0)}\left(Z^T \right)\right) \right|dt}.
        \end{align*}
        \item If, in addition, $f^T=g^T$ for all $T>0$, then, by taking $\sigma^2=1$, we have 
        \begin{align*}
            d_W(F^T,G) &\leq  \E{\left( F^T\right)^3} +\frac{1}{T} \int_{0}^T \int_{t}^T\E{  \sqrt{f^T(t)}\sqrt{f^T(s)} \left| D_{(t,0)}\left( Z_s^T\right) \right| }ds dt \\
            & \quad + \frac{1}{\sqrt{T}} \int_{0}^T\int_{t}^T \E{  \sqrt{f^T(t)}f^T(s) \left|  \delta\left( \z^T D_{(s,0)}\left(Z^T\right)\right) D_{(t,0)}\left( Z_s^T\right)\right| } ds dt\\
            &\quad -\frac{2}{T}\E{\int_{0}^T  \sqrt{f^T(t)} \EE_t\left[\int_t^T \sqrt{f^T(s)} D_{(t,0)}\left(Z^T_s \right) ds \right] dt }.
        \end{align*}
    \end{enumerate}
\end{theorem}

\begin{remark}
Note that although a negative sign appears in the expression of the last term, its actual sign remains undetermined and requires further computations to be fully established. In particular, if $g_s^T \circ \eps_{(t,0)} ^+ > g_s^T $ for any $s \geq t$, then $D_{(t,0)}\left(Z^T_s \right)$ is necessarily negative.
\end{remark}

Before giving the proof of Theorem \ref{thm: main}, we present another theorem which is a special case of Theorem \ref{thm: main}. The latter will deal with the case $g^T$ deterministic which, in the literature, is the most employed case.

\begin{theorem}
    \label{thm : Maj_g_determinist}
    Let $f$ and $g$ such that Assumptions $\ref{assump: dans S}$ and $\ref{assump: initiale}$ hold. 
    Suppose in addition, that $g^T$ is deterministic bounded from below and such that
    $$ \sup\limits_{T>0} \E{\frac{1}{T}\int_{0}^T \frac{f^T(t)}{g^T(t)} dt}< +\infty.$$
    Then, there exists $C>0$ such that for any $\sigma^2 \in \RR_+^*$ and for any $T>0$, 
    $$ d_W(F^T,G) \leq \E{  \left|\sigma^2- \frac{1}{T} \int_{0}^T \frac{f^T(t)}{g^T(t)} dt\right|} +\frac{C}{\sqrt{T}}, \quad G \sim \Ncal(0,\sigma^2). $$
\end{theorem}


\begin{proof}[Proof of Theorem \ref{thm: main}] 
    We divide the proof Theorem \ref{thm: main} in two main parts by regroupig the two first inequality in the same part. Indeed, result $2.$ can be seen as a direct consequence of $1.$ under the assumption $\z^T D_{(t,\theta)}\left( Z^T\right) = \z^T D_{(t,0)}\left( Z^T\right)$. Moreover, $1.$ and $2.$ both consider general Gaussian $\Ncal \left( 0, \sigma^2\right)$ which is not the case for $3.$ where the result holds for standard Gaussian $\Ncal \left( 0, 1\right)$.
    \begin{enumerate}
    \item Let $v\in \mathrsfs{F}_W=\left\{u:\RR \to \RR,\; \textrm{ twice differentiable with } \|u'\|_\infty \leq 1, \; \|u''\|_\infty \leq 2\right\}$. Using the integration by part formula (see Proposition \ref{prop:IPP}), one can get
    $$\E { F^T v(F^T)} = \E{\delta\left(\z^{T} Z^T\right)v(F^T)} = \E { \iint_{\RR_+^2} \z_{(t,\theta)}^{T} Z_t^T D_{(t,\theta)}\left(v(F^T) \right) dtd\theta}.$$
    Besides, Taylor expansion formula gives for any $(t,\theta)\in \RR_+^2$,
    $$ D_{(t,\theta)}(v(F^T)) = v(F^T \circ \eps_{(t,\theta)}^+) - v(F^T) = v'(F^T) D_{(t,\theta)}(F^T) + \frac{1}{2} v''(\hat F^T) \vert D_{(t,\theta)}(F^T) \vert^2,$$
    where $\hat F^T$ be a random element between $F^T \circ \eps_{(t,\theta)}^+$ and $F^T$.\\
    Combining the two equations, it holds that
    \begin{align*}
        \E{F^T v(F^T)} &= \E{ \iint_{\RR_+^2} \z_{(t,\theta)}^{T} Z_t^T v'(F^T)D_{(t,\theta)}\left(F^T \right) dtd\theta} \\
        & \quad + \frac{1}{2} \E {\iint_{\RR_+^2} \z_{(t,\theta)}^{T} Z_t^T v''(\hat F^T)\left|D_{(t,\theta)}\left(F^T \right) \right|^2 dtd\theta}.
    \end{align*}
    We denote by $A_1$ and $A_2$ the right-hand terms of the equation, i.e.
    \begin{align*}
        A_1 &:= \E{\iint_{\RR_+^2} \z_{(t,\theta)}^{T} Z_t^T v'(F^T)D_{(t,\theta)}\left(F^T \right) dtd\theta}\\
        \text{and} \quad A_2 &:= \frac{1}{2} \E { \iint_{\RR_+^2} \z_{(t,\theta)}^{T} Z_t^T v''(\hat F^T)\left|D_{(t,\theta)}\left(F^T \right) \right|^2 dtd\theta}.
    \end{align*}

    In both quantities, an important factor to be considered is the Malliavin derivative of $F^T$.
    To deal with this term, one can use Proposition \ref{prop: heisenberg} combining with the property of the Malliavin calculus to get
    \begin{align}
        D_{(t,\theta)}\left( F^T\right) &= D_{(t,\theta)}\left( \delta \left( \z^T Z^T \right) \right) \nonumber\\
        &= \z_{(t,\theta)}^T Z_t^T + \delta \left( D_{(t,\theta)}\left( \z^T\right) \left[ Z^T + D_{(t,\theta)}\left( Z^T\right)\right] \right) \label{eq: Heis et malliavin} \\ 
        &\quad + \delta \left( \z^T D_{(t,\theta)}\left(Z^T \right)\right). \nonumber
    \end{align}
    Thus, we substitute this equality in $A_1$ to obtain
    \begin{align*}
        A_1 &= \E { \iint_{\RR_+^2} \z_{(t,\theta)}^{T} \left|Z_t^T\right|^2 v'(F^T) dtd\theta}\\
        &\quad + \E { \iint_{\RR_+^2} \z_{(t,\theta)}^{T} Z_t^T v'(F^T)\delta \left( D_{(t,\theta)}\left( \z^T\right) \left[ Z^T + D_{(t,\theta)}\left( Z^T\right)\right] \right) dtd\theta}\\
        &\quad + \E { \iint_{\RR_+^2} \z_{(t,\theta)}^{T} Z_t^T v'(F^T)\delta \left( \z^T D_{(t,\theta)}\left(Z^T \right)\right) dtd\theta}.
    \end{align*}

    Moreover, Lemma \ref{lem: magique} in Section \ref{sec:technical} gives that
    $$\E { \iint_{\RR_+^2} \z_{(t,\theta)}^{T} Z_t^T v'(F^T)\delta \left( D_{(t,\theta)}\left( \z^T\right) \left[ Z^T + D_{(t,\theta)}\left( Z^T\right)\right] \right) dtd\theta} =0$$
    and so $A_1$ is reduced to the following:
    \begin{align}
        A_1 &= \E { \iint_{\RR_+^2} \z_{(t,\theta)}^{T} \left|Z_t^T\right|^2 v'(F^T) dtd\theta}+ \E{ \iint_{\RR_+^2} \z_{(t,\theta)}^{T} Z_t^T v'(F^T)\delta \left( \z^T D_{(t,\theta)}\left(Z^T \right)\right) dtd\theta} \nonumber \\
        &= \E{ \frac{v'(F^T)}{T} \int_{0}^T \frac{f^T(t)}{g^T(t)} dt}+ \E{ \iint_{\RR_+^2} \z_{(t,\theta)}^{T} Z_t^T v'(F^T)\delta \left( \z^T D_{(t,\theta)}\left(Z^T \right)\right) dtd\theta}.
        \label{eq: maj_A_1}
    \end{align}
    We can use the same methodology on $A_2$ and so we get
    \begin{align}
        A_2 &= \frac{1}{2}\E{ \iint_{\RR_+^2} \z_{(t,\theta)}^{T} \left(Z_t^T \right)^3 v''(\hat F^T) dtd\theta} \nonumber \\
        &\quad +\frac{1}{2} \E{ \iint_{\RR_+^2} \z_{(t,\theta)}^{T} Z_t^T v''(\hat F^T)\left|\delta \left( \z^T D_{(t,\theta)}\left(Z^T \right)\right) \right|^2 dtd\theta}  \label{eq: maj_A_2}\\
        & \quad + \E{ \iint_{\RR_+^2} \z_{(t,\theta)}^{T} \left|Z_t^T \right|^2 v''(\hat F^T)\delta \left( \z^T D_{(t,\theta)}\left(Z^T \right)\right) dtd\theta}. \nonumber 
    \end{align}
    Therefore, concatenating \eqref{eq: maj_A_1} and \eqref{eq: maj_A_2}, we have
    \begin{align*}
        \E{ \sigma^2v'(F^T)-F^T v(F^T)} & =\E{ v'(F^T) \left(\sigma^2- \frac{1}{T} \int_{0}^T \frac{f^T(t)}{g^T(t)} dt\right)}\\
        & \quad -\E{ \iint_{\RR_+^2} \z_{(t,\theta)}^{T} Z_t^T v'(F^T)\delta \left( \z^T D_{(t,\theta)}\left(Z^T \right)\right) dtd\theta} \\
        & \quad -\frac{1}{2}\E{ \iint_{\RR_+^2} \z_{(t,\theta)}^{T} \left(Z_t^T \right)^3 v''(\hat F^T) dtd\theta}\\
        & \quad -\frac{1}{2} \E{ \iint_{\RR_+^2} \z_{(t,\theta)}^{T} Z_t^T v''(\hat F^T)\left|\delta \left( \z^T D_{(t,\theta)}\left(Z^T \right)\right) \right|^2 dtd\theta}\\
        & \quad - \E{ \iint_{\RR_+^2} \z_{(t,\theta)}^{T} \left|Z_t^T \right|^2 v''(\hat F^T)\delta \left( \z^T D_{(t,\theta)}\left(Z^T \right)\right) dtd\theta}.
    \end{align*}
    Since $v\in \mathrsfs{F}_W$, we have $\Vert v' \Vert_{\infty} \leq 1$ and $\Vert v'' \Vert_{\infty} \leq 2$, and so we obtain
    \begin{align*}
        \E{ \sigma^2v'(F^T)-F^T v(F^T)} & \leq \E{  \left|\sigma^2- \frac{1}{T} \int_{0}^T \frac{f^T(t)}{g^T(t)} dt\right|}+\E{ \int_{\RR_+^2} \z_{(t,\theta)}^{T} \left|Z_t^T \right|^3  dtd\theta}\\
        & \quad +\E{ \iint_{\RR_+^2} \z_{(t,\theta)}^{T} Z_t^T \left|\delta \left( \z^T D_{(t,\theta)}\left(Z^T \right)\right) \right| dtd\theta} \\
        & \quad + \E{ \iint_{\RR_+^2} \z_{(t,\theta)}^{T} Z_t^T \left|\delta \left( \z^T D_{(t,\theta)}\left(Z^T \right)\right) \right|^2 dtd\theta}\\
        & \quad +2 \E{ \iint_{\RR_+^2} \z_{(t,\theta)}^{T} \left|Z_t^T \right|^2 \left|\delta \left( \z^T D_{(t,\theta)}\left(Z^T \right)\right) \right|dtd\theta}.
    \end{align*}
    Replacing $Z^T$ and $\z^T$ by their definitions, we get the following:
    \begin{align*}
        \E{ \sigma^2v'(F^T)-F^T v(F^T)} & \leq \E{  \left|\sigma^2- \frac{1}{T} \int_{0}^T \frac{f^T(t)}{g^T(t)} dt\right|}+\frac{1}{T^{3/2}} \E{\int_{0}^T \frac{f^T(t)}{(g^T(t))^{3/2}}  dt}\\
        & \quad +\E{ \iint_{\RR_+^2} \z_{(t,\theta)}^{T} Z_t^T \left|\delta \left( \z^T D_{(t,\theta)}\left(Z^T \right)\right) \right| dtd\theta} \\
        & \quad + \E{ \iint_{\RR_+^2} \z_{(t,\theta)}^{T} Z_t^T \left|\delta \left( \z^T D_{(t,\theta)}\left(Z^T \right)\right) \right|^2 dtd\theta}\\
        & \quad +2 \E{ \iint_{\RR_+^2} \z_{(t,\theta)}^{T} \left|Z_t^T \right|^2 \left|\delta \left( \z^T D_{(t,\theta)}\left(Z^T \right)\right) \right|dtd\theta}.
    \end{align*}
    Given that this holds for any $v\in \mathrsfs{F}_W$ (defined in \eqref{eq: def_FW}), we can take the supremum of the left-hand term, yielding the result.

    \item We now suppose that $\z^T D_{(t,\theta)}\left( Z^T\right) = \z^T D_{(t,0)}\left( Z^T\right)$ for any $(t,\theta)\in \RR_+^2$. Thanks to this assumption, we can integrate with respect to $\theta$ in the previous inequality and so we get
    \begin{align*}
        d_W(F^T,G) &\leq \E{  \left|\sigma^2- \frac{1}{T} \int_{0}^T \frac{f^T(t)}{g^T(t)} dt\right|} +\frac{1}{T^{3/2}} \E{\int_{0}^T \frac{f^T(t)}{(g^T(t))^{3/2}}  dt}\\
            & \quad + \frac{1}{\sqrt{T}}\E{ \int_{0}^T \frac{f^T(t)}{\sqrt{g^T(t)}}\left|\delta \left( \z^T D_{(t,0)}\left(Z^T \right)\right) \right| dt} \\
            & \quad +\frac{1}{\sqrt{T}}\E{ \int_{0}^T \frac{f^T(t)}{\sqrt{g^T(t)}} \left|\delta \left( \z^T D_{(t,0)}\left(Z^T \right)\right) \right|^2 dt}\\
            & \quad + \frac{2}{T}\E{ \int_{0}^T \frac{f^T(t)}{g^T(t)} \left|\delta \left( \z^T D_{(t,0)}\left(Z^T \right)\right) \right|dt}.
    \end{align*}

    \item Suppose that $f^T=g^T$(note that this assumption implies that $\sigma^2=1$ since Var$\left[ F^T\right] =1$). In order to prove $3.$, we start by making use of Lemma \ref{lem: ineq 2} which give for $v\in \mathrsfs{F}_W$ and $T>0$,
    \begin{align*}
        \E { F^T v(F^T)-v'\left( F^T\right)} 
        & \leq  \E{\left( F^T\right)^3} + \E{\left(v'\left(F^T \right)-2F^T \right)\iint_{\RR^2_+}  \z^T_{(t,\theta)} Z^T \delta \left( D_{(t,\theta)}\left( \z^T Z^T\right)\right) dt d\theta}.
    \end{align*}
    We now substitutes $\delta \left( D_{(t,\theta)}\left( \z^T Z^T\right)\right)$ by $\delta \left( \z^T D_{(t,\theta)}\left(  Z^T\right)\right)$ according to Lemma \ref{lem: magique} and so we have
    \begin{align*}
        \E { F^T v(F^T)-v'\left( F^T\right)} & \leq  \E{\left( F^T\right)^3} + \E{\left(v'\left(F^T \right)-2F^T \right)\iint_{\RR^2_+}  \z^T_{(t,\theta)} Z^T \delta \left( \z^T D_{(t,\theta)}\left( Z^T\right)\right) dt d\theta}\\
         &\leq  \E{\left( F^T\right)^3} + \frac{1}{\sqrt{T}}\E{\left(v'\left(F^T \right)-2F^T \right)\int_{0}^T  \sqrt{f^T(t)}\delta \left(\z^T D_{(t,0)}\left(  Z^T\right)\right) dt }.
    \end{align*}
    Let us denote by $B_1$ and $B_2$ the following quantities:
    \begin{align*}
        B_1 &= \frac{1}{\sqrt{T}}\E{v'\left(F^T \right)\int_{0}^T  \sqrt{f^T(t)}\delta \left(\z^T D_{(t,0)}\left(  Z^T\right)\right) dt }\\
        \text{and} \quad B_2 &= -\frac{2}{\sqrt{T}}\E{F^T\int_{0}^T  \sqrt{f^T(t)}\delta \left(\z^T D_{(t,0)}\left(  Z^T\right)\right) dt }.
    \end{align*}
    We first deal with $B_2$ by noticing that 
    $$B_2 = -\frac{2}{\sqrt{T}}\E{\int_{0}^T  \sqrt{f^T(t)}\delta\left( \z^T Z^T \right)\delta \left(\z^T D_{(t,0)}\left(  Z^T\right)\right) dt }. $$
   Thereupon, we may employ Lemma $\ref{lem: prod_div}$ in order to get
    \begin{align}
        B_2= -\frac{2}{T}\E{\int_{0}^T  \sqrt{f^T(t)} \EE_t\left[\int_t^T \sqrt{f^T(s)} D_{(t,0)}\left(Z^T_s \right) ds \right] dt }. \label{eq: upperbound_B_2}
    \end{align} 
    On the other hand, we make use of the integration by part formula (see Proposition \ref{prop:IPP}) on $B_1$ to compute
    \begin{align*}
        B_1 &= \frac{1}{\sqrt{T}} \int_{0}^T \E{ \iint_{\RR_+^2}  D_{(s,\rho)}\left(\sqrt{f^T(t)} v'\left(F^T \right) \right)\z_{(s,\rho)}^T D_{(t,0)}\left( Z_s^T\right) ds d\rho} dt\\
         &= \frac{1}{\sqrt{T}} \int_{0}^T \E{ \iint_{\RR_+^2} \sqrt{f^T(t)} D_{(s,\rho)}\left(v'\left(F^T \right) \right)\z_{(s,\rho)}^T D_{(t,0)}\left( Z_s^T\right) ds d\rho} dt .
    \end{align*}
    Moreover, thanks to Taylor expansion formula, there exists $\hat{F}$ between $F^T$ and $F^T\circ \eps_{(s,\rho)}^+$ such that
    $$ D_{(s,\rho)}\left(v'\left(F^T \right) \right)  = D_{(s,\rho)}\left(F^T \right) v''(\hat{F}).$$
    And so
    \begin{align*}
        B_1 &= \frac{1}{\sqrt{T}} \int_{0}^T \E{ \iint_{\RR_+^2} \sqrt{f^T(t)}\z_{(s,\rho)}^T v''(\hat{F}^T)  D_{(s,\rho)}\left(F^T \right)D_{(t,0)}\left( Z_s^T\right) ds d\rho} dt .
    \end{align*}
    We now employ the same methodology as for $A_1$ and we apply Equation \eqref{eq: Heis et malliavin} followed by Lemma \ref{lem: magique} to get
    \begin{align*}
        B_1 &= \frac{1}{\sqrt{T}} \int_{0}^T \E{ \iint_{\RR_+^2} \sqrt{f^T(t)}\z_{(s,\rho)}^T v''(\hat{F}^T)  Z^T_s D_{(t,0)}\left( Z_s^T\right) ds d\rho} dt \\
        & \quad + \frac{1}{\sqrt{T}} \int_{0}^T \E{ \iint_{\RR_+^2} \sqrt{f^T(t)}\z_{(s,\rho)}^T v''(\hat{F}^T)  \delta\left( \z^T D_{(s,0)}\left(Z^T\right)\right) D_{(t,0)}\left( Z_s^T\right) ds d\rho} dt\\
        &= \frac{1}{T} \int_{0}^T \int_{t}^T\E{  \sqrt{f^T(t)}\sqrt{f^T(s)}  v''(\hat{F}^T)   D_{(t,0)}\left( Z_s^T\right) }ds dt \\
        & \quad + \frac{1}{\sqrt{T}} \int_{0}^T\int_{t}^T \E{  \sqrt{f^T(t)}f^T(s) v''(\hat{F}^T)  \delta\left( \z^T D_{(s,0)}\left(Z^T\right)\right) D_{(t,0)}\left( Z_s^T\right) } ds dt.
    \end{align*}
    Using the boundedness of $v''$, we obtain
    \begin{align}
        B_1 &\leq \frac{1}{T} \int_{0}^T \int_{t}^T\E{  \sqrt{f^T(t)}\sqrt{f^T(s)} \left| D_{(t,0)}\left( Z_s^T\right) \right| }ds dt \nonumber \\
        & \quad + \frac{1}{\sqrt{T}} \int_{0}^T\int_{t}^T \E{  \sqrt{f^T(t)}f^T(s) \left|  \delta\left( \z^T D_{(s,0)}\left(Z^T\right)\right) D_{(t,0)}\left( Z_s^T\right)\right| } ds dt. \label{eq: upperbound_B_1}
    \end{align}
    Combining the upper bound \eqref{eq: upperbound_B_1} with \eqref{eq: upperbound_B_2}, we get
    \begin{align*}
        &\E { F^T v(F^T)-v'\left( F^T\right)} \\
        & \leq  \E{\left( F^T\right)^3} +\frac{1}{T} \int_{0}^T \int_{t}^T\E{  \sqrt{f^T(t)}\sqrt{f^T(s)} \left| D_{(t,0)}\left( Z_s^T\right) \right| }ds dt \\
        & \quad + \frac{1}{\sqrt{T}} \int_{0}^T\int_{t}^T \E{  \sqrt{f^T(t)}f^T(s) \left|  \delta\left( \z^T D_{(s,0)}\left(Z^T\right)\right) D_{(t,0)}\left( Z_s^T\right)\right| } ds dt\\
        &\quad -\frac{2}{T}\E{\int_{0}^T  \sqrt{f^T(t)} \EE_t\left[\int_t^T \sqrt{f^T(s)} D_{(t,0)}\left(Z^T_s \right) ds \right] dt } .
    \end{align*}
    The latter is true for any $v\in \mathrsfs{F}_W$, so passing to the supremum, we have the expected result.
    \end{enumerate}
\end{proof}

We now give the proof of Theorem \ref{thm : Maj_g_determinist}.
\begin{proof}[Proof of Theorem \ref{thm : Maj_g_determinist}]
Since $g^T$ is a deterministic function, its Malliavin derivative is null. Thereupon, for any $(t,\theta)\in \RR_+^2$, we have 
    $$D_{(t,\theta)} Z^T = 0.$$
Applying Theorem \ref{thm: main}, we obtain:
    \begin{align*}
        d_W(F^T,G) \leq \E{\left| \sigma^2- \frac{1}{T} \int_{0}^T \frac{f^T(t)}{g^T(t)} dt\right|} + +\frac{1}{T^{3/2}} \E{\int_{0}^T \frac{f^T(t)}{(g^T(t))^{3/2}}  dt}.
    \end{align*}
    Moreover, using the fact that $g^T$ is bounded from below by some constant $\mu >0$ combining with the fact that $\frac{1}{T}\E{\int_0^T \frac{f^T(t)}{g^T(t)} dt} \leq \overline{C}$, with $\overline{C}>0$,  we have
    \begin{align*}
        d_W(F^T,G) &\leq \E{\left| \sigma^2- \frac{1}{T} \int_{0}^T \frac{f^T(t)}{g^T(t)} dt\right|} + \frac{1}{\sqrt{\mu}\sqrt{T}} \E{\frac{1}{T}\int_{0}^T \frac{f^T(t)}{g^T(t)} dt}\\
        &\quad \leq  \E{\left| \sigma^2- \frac{1}{T} \int_{0}^T \frac{f^T(t)}{g^T(t)} dt\right|} + \frac{\overline{C}/\sqrt{\mu}}{\sqrt{T}}.
    \end{align*}
\end{proof}

Our third main result is a generalisation of Theorem \ref{thm : Maj_g_determinist}. It gives a general quantification for counting processes which can be represented using a Poisson imbedding.

\begin{proposition}
\label{prop:main2}
Let $T >0$ and denote by $\Delta_T$ a coefficient such that $\left|\Delta_T\right|^2 \geq \mu T$ for some constant $\mu \in \RR_+^*$.
Let $\lambda:=(\lambda_t)_{t\geq 0}$ be a $\FF^N$-predictable non-negative process such that 
\begin{enumerate}
    \item $ \E{\left(\int_0^T \lambda_t dt\right)^2}<+\infty \quad \text{and} \quad \E{\frac{1}{\left|\Delta_T\right|^2} \int_0^T \lambda_t dt}< C $;
    \item for any $(t,\theta) \in \RR_+^2$, $\ind{\theta \geq \lambda_t} \lambda\circ \eps_{(t,\theta)}^+ = \ind{\theta \geq \lambda_t} \lambda$.
\end{enumerate}
Consider the counting process $H:=(H_t)_{t\geq 0}$ defined as : 
$$ H_t:=\iint_{(0,t]\times \RR_+} \ind{\theta \leq \lambda_s} N(ds,d\theta), \quad t\geq 0 $$
so that $H$ has intensity $\lambda$ (in the sense that the process $M:=(M_t)_{t\geq 0}$ defined as $M_t:=H_t - \int_0^t \lambda_s ds$ is a $\FF^N$-martingale).\\
Let $\sigma^2>0$ and $s_T:= \E{\left|\sigma^2 - \frac{1}{\Delta_T^2} \int_0^T \lambda_t dt \right|}$. Then 
$$ d_W\left(\frac{H_T-\int_0^T \lambda_t dt}{\Delta_T},G\right) \leq C\left(\frac{1}{\Delta_T} \vee s_T\right); \quad G \sim \Ncal(0,\sigma^2). $$
\end{proposition}

\begin{proof}
Repeat the proof of Theorem \ref{thm : Maj_g_determinist} with $f^T:=\lambda$ and $g^T:=\frac{\Delta_T^2}{T}$.
\end{proof}


\section{Applications to non-linear Hawkes processes}
\label{sec: App Hawkes}
Our investigation, outlined in Section \ref{sec: Notation}, relies on analyzing the Malliavin operator concerning quantities $\z$ and $Z$. We extend these analyses to cases where $f$ and $g$ are defined within the framework of a Hawkes process. Initially, we introduce the concepts of stationary Hawkes processes and those with an empty history. Subsequently, we draw comparisons between these two types of Hawkes processes, particularly when the kernel is non-negative. Furthermore, we present results from Malliavin's calculus applied to Hawkes processes (either with empty history or the stationary version). Lastly, we apply Theorem \ref{thm : Maj_g_determinist} within the context of Hawkes processes to establish an explicit convergence rate of $F^T$ to a centered Gaussian distribution, under certain assumptions regarding $h$ or $\phi$. It is important to note that all our findings are pertinent to non-linear Hawkes processes.

\subsection{Non-linear Hawkes processes}
The goal of this section is to obtain convergence rate between a Gaussian distribution and $F^T$ in the context of non-linear Hawkes processes. To do so, it is important to define the stationary Hawkes processes since the variance of $F^T$ converges towards the expectation of the intensity of a stationary Hawkes process.
We naturally begin this section with a definition of Hawkes processes with $(\Upsilon= - \infty)$ or without $(\Upsilon>-\infty)$ past history. 

\begin{definition}[Hawkes process, \cite{bremaud_stability_1996}]
\label{def:Hawkes}
Let $(\Omega, \FF, \PP)$ be a general probability space.
Let $\mu\in \RR$, $\phi:\RR_+ \to \RR$ and $h: \RR \to \RR_+$. A standard Hawkes process $H^{(\Upsilon)}:=(H^{(\Upsilon)}_t)_{t \geq 0}$ with parameters $h$ and $\phi$ is a counting process such that   
\begin{itemize}
\item[(i)] $H^{(\Upsilon)}_0=0,\quad \PP-a.s.$,
\item[(ii)] its ($\FF$-predictable) intensity process is given by
$$\lambda^{(\Upsilon)}_t:=h \left(\mu + \int_{(\Upsilon,t)} \phi(t-s) dH^{(\Upsilon)}_s\right), \quad t\geq \Upsilon,$$
that is for any $\Upsilon \leq s \leq t $ and $A \in \Fcal_s$,
$$ \E{\1_A (H^{(\Upsilon)}_t-H^{(\Upsilon)}_s)} = \E{\int_{(s,t]} \1_A \lambda^{(\Upsilon)}_r dr }.$$
\end{itemize}
\end{definition}

For the sake of this paper, we make use of the following notation:
\begin{notation}
   \begin{itemize}
    \item[$\bullet$] We will indicate by $H^{(\Upsilon)} =H^{(-\infty)} =:  H^\infty$ and $\lambda^{(\Upsilon)} = \lambda^{(-\infty)} = \lambda^{\infty}$ when we deal with stationary Hawkes processes. 
    \item[$\bullet$] We will indicate by $H^{(\Upsilon)} =H^{(0)} =:  H$ and $\lambda^{(\Upsilon)} = \lambda^{(0)} = \lambda$ when we deal with Hawkes processes with empty past history. 
\end{itemize} 
\end{notation}

An important version of Hawkes processes are the linear ones which were introduced by Hawkes \cite{hawkes_spectra_1971}. 
\begin{definition}[linear Hawkes process] We say that a Hawkes process is linear when $\mu>0$ and $h$ is of the form $ h(x) = x$.
\end{definition}

Inspired by \cite{bremaud_stability_1996}, we introduce the following technical assumption on $\phi$ and $h$:
\begin{assumption} \label{assump: noyau}
    We assume that: \begin{enumerate}
        \item $\phi : \RR_+ \to \RR$ is locally bounded and such that
        $$\left\Vert \phi\right\Vert_{1}:= \int_0^{\infty} |\phi(t)| dt < +\infty \quad \text{and} \quad m:= \int_0^{\infty} t |\phi(t) | dt < \infty$$
        \item $h: \RR \to \RR_+$ is positive and $\alpha$-Lipschitz, i.e. for any $(x,y)\in \RR$, $\left| h(x) - h(y) \right|\leq 
        \alpha \left| x-y \right|$  with $\alpha \left\Vert \phi\right\Vert_{1} < 1$.
    \end{enumerate}
\end{assumption}

\paragraph{Stationary Hawkes processes}~\\
The definition provided before is not the one used in this paper. Instead, we opt for representing the Hawkes process through a solution of a SDE. This representation can be found in \cite{bremaud_stability_1996} and constitute our chosen approach. Specifically, they established the following theorem and corollary:

\begin{theorem}
    \label{th:BM}
    Under Assumption \ref{assump: noyau}, the SDE below admits a unique $\FF^N$-predictable solution $\lambda^{\infty}$. 
    \begin{equation}
    \label{eq:Int_Hawkes_station}
    \lambda^{\infty}_t = h \left(\mu + \iint_{(-\infty,t) \times \RR_+} \phi(t-u)\ind{\theta \leq \lambda_u^{\infty}} N(du,d\theta) \right) ,\quad t \in \RR_+ .
    \end{equation}
\end{theorem}

\begin{corollary}
    Suppose Assumption \ref{assump: noyau} holds. We denote by $\lambda^{\infty}$ the unique solution of \eqref{eq:Int_Hawkes_station} and by $H^{\infty}$ his associated counting process, i.e.  
    $$H^{\infty}_t := \iint_{(-\infty,t]\times \RR_+} \1_{\{\theta \leq \lambda^{\infty}_s\}} N(ds,d\theta), \quad t\in \RR_+.$$
    Then, $H^{\infty}$ is a Hawkes process with finite average intensity and whose intensity is $\lambda^{\infty}$.
\end{corollary}

\paragraph{Hawkes processes with empty history}~\\
One specific case is the Hawkes process with an empty history. This entails a Hawkes process $(H,\lambda)$ where $H(\RR^*_{-}) = \emptyset$. The representation of such processes as solutions to SDEs has been explored by Hillairet et al. in \cite{hillairet_expansion_2023}, albeit for linear Hawkes processes. Here, we extend their findings to encompass non-linear Hawkes processes.

\begin{theorem}
    \label{th:HRR}
    Under Assumption \ref{assump: noyau}, the SDE below admits a unique $\FF^N$-predictable solution $\lambda$. 
    \begin{equation}
    \label{eq:Int_Hawkes}
    \lambda_t = h \left(\mu + \iint_{(0,t)\times \RR_+} \phi(t-u)\ind{\theta \leq \lambda_u} dN(u,\theta) \right) ,\quad t \in \RR_+ .
    \end{equation}
\end{theorem}

\begin{remark}
    We refer to Appendix \ref{app : volterra} for the existence. More precisely, we refer to Proposition \ref{prop: existence}.
\end{remark}

\begin{corollary}
    Suppose that Assumption \ref{assump: noyau} holds. We consider $(H,\lambda)$ the unique solution to the SDE 
    \begin{equation}
    \left\lbrace
    \begin{array}{l}
        H_t = \iint_{(0,t] \times \RR_+} \ind{\theta \leq \lambda_s} N(ds,d\theta),\quad t \in \RR_+ \\
        \lambda_t = h \left(\mu + \int_{(0,t)} \phi(t-u) d H_u \right),\quad t \in \RR_+
    \end{array}
    \right. .
    \end{equation}
    We set $\FF^H$ the natural filtration of $H$. Then $H$ is a Hawkes process in the sense of Definition \ref{def:Hawkes}.
\end{corollary}
\begin{proof}
    The proof is similar to the one done in \cite{hillairet_expansion_2023}.
\end{proof}

\paragraph{Comparaison of Hawkes processes}~\\
In this paragraph, we give a comparaison between two different Hawkes processes, denoted $H^{(1)}$ and $H^{(2)}$, defined with the same Poisson random measure. In particular, we will denote by $\lambda^{(1)}$ and $\lambda^{(2)}$ their respective intensities, i.e. for $t\in \RR_+$,
\begin{align*}
    \lambda^{(1)}_t &= h \left( \mu^{(1)}_t + \iint_{(0,t)\times \RR_+} \phi\left( t-s\right) \ind{\theta \leq \lambda^{(1)}_s} N(ds,d\theta)\right);\\
    \lambda^{(2)}_t &= h \left( \mu^{(2)}_t + \iint_{(0,t)\times \RR_+} \phi\left( t-s\right) \ind{\theta \leq \lambda^{(2)}_s} N(ds,d\theta)\right),
\end{align*}
where $\mu^{(1)}$ and $\mu^{(2)}$ are both $\Fcal_0^N$-measurable functions.
\begin{notation}
\label{notation: xi generaux}
    We denote by $(\xi^{(1)}_t)_{t\in \RR_+}$ and by $(\xi^{(2)}_t)_{t\in \RR_+}$ the following
    \begin{align*}
        \xi^{(1)}_t &:= \iint_{(0,t)\times \RR_+} \phi(t-s) \1_{\theta \leq \lambda^{(1)}_s} N(ds,d\theta), \quad t\in \RR_+; \\
        \xi^{(2)}_t &:= \iint_{(0,t)\times \RR_+} \phi(t-s) \1_{\theta \leq \lambda^{(2)}_s} N(ds,d\theta), \quad t\in \RR_+.
    \end{align*}
\end{notation}
With these notations on hand, we can find a relation between $\lambda_t^{(1)}$ and $\lambda_t^{(2)}$. This is detailed in Proposition \ref{prop: comparaison hawkes}.

\begin{proposition}
\label{prop: comparaison hawkes}
    Suppose that $\phi$ is a non-negative function and that for all $t\in \RR_+$, $\mu^{(1)}_t \geq \mu^{(2)}_t$. Then, for $t\in \RR_+$,
    \begin{enumerate}
        \item if $h$ is non-decreasing, $\lambda_t^{(1)} \geq \lambda_t^{(2)}$;
        \item if $h$ is non-increasing, $\lambda_t^{(1)} \leq \lambda_t^{(2)}$.
    \end{enumerate}
\end{proposition}
\begin{proof}
    Suppose that $h$ is non-decreasing (resp. non-increasing) and let us prove by induction the result.\\
    For $n\in \NN^*$, we denote by $\left(\tau_n\right)_{n\in \NN}$ the following sequence:  
    $$\left\{ \begin{array}{ll}
         \tau_1 &:= \min\{ \tau_1^{(1)}, \tau_1^{(2)}\};\\
         \tau_n &:= \inf \left\{ \tau >0 \mid (\tau, \theta) \in (\tau_{n-1}, +\infty) \times [0, \lambda_\tau^{(1)}\vee \lambda_\tau^{(2)}], N\{ (t,\theta)\} =1\right\}, n\geq 2.
    \end{array} \right.$$
    Here, $\left( \tau_i^{(1)}\right)_{i\in \NN^*}$ (resp. $\left( \tau_i^{(2)}\right)_{i\in \NN^*}$) is the sequence of positive time-jump of $H^{(1)}$ (resp. $H^{(2)}$).\\
    For $t\in (0, \tau_1]$, we have that $\xi^{(j)}= 0$ for $j\in \{1,2\}$. Hence, we get
    $$ \mu_t^{(1)} + \xi_t^{(1)} = \mu_t^{(1)} \geq \mu_t^{(2)} =  \mu_t^{(2)}+\xi^{(2)}_t .$$
    Since $h$ is non-decreasing (resp. non-increasing) we obtain $\lambda^{(1)} \geq \lambda^{(2)}$ (resp. $\lambda^{(1)} \leq \lambda^{(2)}$) on $(0,\tau_{1}]$.\\
    Now, let $k\in \NN^*$ and suppose that $\lambda^{(1)} \geq \lambda^{(2)}$ (resp. $\lambda^{(1)} \leq \lambda^{(2)}$) on $(0,\tau_k]$.
    For $t\in (0, \tau_{k+1}]$, we have 
    $$ \mu_t^{(j)}+\xi_t^{(j)} = \mu^{(j)}_t + \iint_{(0,\tau_{k}]\times \RR_+} \phi(t-s) \1_{\theta \leq \lambda^{(j)}_s} N(ds,d\theta), \quad j\in \{ 1,2\}.$$
    Since $\lambda^{(1)} \geq \lambda^{(2)}$ (resp. $\lambda^{(1)} \leq \lambda^{(2)}$) on $(0,\tau_k]$, we have 
    \begin{align*}
        \mu_t^{(1)}+\xi_t^{(1)} &=\mu^{(1)}_t + \iint_{(0,\tau_k]\times \RR_+} \phi(t-s) \1_{\theta \leq \lambda^{(1)}_s} N(ds,d\theta) \\
        &\geq \mu^{(2)}_t + \iint_{(0,\tau_k]\times \RR_+} \phi(t-s) \1_{\theta \leq \lambda^{(2)}_s} N(ds,d\theta) = \mu_t^{(2)}+\xi^{(2)}_t.
    \end{align*}
    Since $h$ is non-decreasing (resp. non-increasing) we obtain $\lambda^{(1)} \geq \lambda^{(2)}$ (resp. $\lambda^{(1)} \leq \lambda^{(2)}$) on $(0,\tau_{k+1}]$.
\end{proof}

Note that Proposition \ref{prop: comparaison hawkes} allows us to compare the stationary version of a Hawkes process $H^{\infty}$ with a Hawkes process with empty history $H$. To do so, we introduce the following notation:
\begin{notation}
\label{notation: xi}
    We denote by $(\xi_t)_{t\in\RR_+}$ and by $(\xi^{\infty}_t)_{t\in \RR}$ the following
    \begin{align*}
        \xi_t &:= \iint_{(0,t)\times \RR_+} \phi(t-s) \1_{\theta \leq \lambda_s} N(ds,d\theta), \quad t\in \RR_+; \\
        \xi^{\infty}_t &:= \iint_{(-\infty,t)\times \RR_+} \phi(t-s) \1_{\theta \leq \lambda^{\infty}_s} N(ds,d\theta), \quad t\in \RR.
    \end{align*}
    Here, $\lambda$ and $\lambda^{\infty}$ are respectively the intensities of Hawkes processes as defined in \eqref{eq:Int_Hawkes} and \eqref{eq:Int_Hawkes_station}.
\end{notation}

\begin{corollary}
\label{prop: comparaison hawkes stat}
    Suppose that $\phi$ is a non-negative function. Then, for $t\in \RR_+$,
    \begin{enumerate}
        \item if $h$ is non-decreasing, $\lambda_t^{\infty} \geq \lambda_t$;
        \item if $h$ is non-increasing, $\lambda_t^{\infty} \leq \lambda_t$.
    \end{enumerate}
\end{corollary}
\begin{proof}
    Apply Proposition \ref{prop: comparaison hawkes} with $\lambda^{(1)} = \lambda^{\infty}$, $\lambda^{(2)} = \lambda$, $\mu^{(1)}_t=\mu+ \iint_{(-\infty,0]\times \RR_+} \phi(t-s) \1_{\theta \leq \lambda^{\infty}_s}N(ds, d\theta)$ and $\mu^{(2)}=\mu$.
\end{proof}

\subsection{Elements of Malliavin calculus  for non-linear Hawkes processes in continuous time}
\label{subsec: Malliavin}

Following the same lines as \cite{hillairet_malliavin-stein_2022}, we develop here the effect of Malliavin calculus for general Hawkes processes. In this part, we decide to use a non-linear Hawkes process with empty history. Note that the results can also be use for a stationary Hawkes process, that is why we use the same notation as Definition \ref{def:Hawkes}.
Note that these results ensure that hypotheses of Theorem \ref{thm: main} are indeed verified, and thus it can be applied.

\begin{lemma}
    Let $t \geq \Upsilon$ and $(\theta, \theta_0) \in \RR^2_+$, it holds that :
    \begin{align*}
    &\ind{\theta \leq \lambda^{(\Upsilon)}_t}\ind{\theta_0 \leq \lambda^{(\Upsilon)}_t} \left( H^{(\Upsilon)}_s \circ \eps_{(t,\theta)}^+ , \lambda^{(\Upsilon)}_s\circ \eps_{(t,\theta)}^+ \right)_{s\geq \Upsilon} \\
    &=\ind{\theta \leq \lambda^{(\Upsilon)}_t}\ind{\theta_0 \leq \lambda^{(\Upsilon)}_t} \left( H^{(\Upsilon)}_s \circ \eps_{(t,\theta_0)}^+ , \lambda^{(\Upsilon)}_s\circ \eps_{(t,\theta_0)}^+ \right)_{s\geq \Upsilon} .
    \end{align*}
\end{lemma}

\begin{proof}
    During the proof, we work on $\{ \theta_0 \leq \lambda^{(\Upsilon)}_t\}$ for a fixed $t \geq \Upsilon$.
    For any $s\geq t$, we have
    $$H^{(\Upsilon)}_s \circ \eps_{(t,\theta_0)}^+ = H^{(\Upsilon)}_{t^-} + \ind{\theta_0 \leq \lambda^{(\Upsilon)}_t} + \iint_{(t,s]\times \RR_+} \ind{\eta \leq \lambda^{(\Upsilon)}_u \circ \eps_{(t,\theta_0)}^+ } N(du,d\eta )$$
    and $H^{(\Upsilon)}_s \circ \eps_{(t,\theta_0)}^+ = H^{(\Upsilon)}_s$, for $s<t$. We refer to \cite{hillairet_malliavin-stein_2022} for a detailed proof. In the similar fashion, $\lambda^{(\Upsilon)}_s \circ \eps_{(t,\theta_0)}^+ = \lambda^{(\Upsilon)}_s$ for $s<t$ and for any $s\geq t$, we have
    \begin{align*}
        \lambda^{(\Upsilon)}_s \circ \eps_{(t,\theta_0)}^+ &= h \left( \mu + \int_{(0,t)} \phi (s-u) dH^{(\Upsilon)}_u + \int_{[t,s)} \phi (s-u) dH^{(\Upsilon)}_u \right)\circ \eps_{(t,\theta_0)}^+ \\
        &= h \left( \mu + \int_{(0,t)} \phi (t-u) dH^{(\Upsilon)}_u + \phi(s-t) +\int_{(t,s)} \phi (s-u) d\left(H^{(\Upsilon)}_u\circ \eps_{(t,\theta_0)}^+ \right)\right)
    \end{align*}
    Since $\left( H^{(\Upsilon)} \circ \eps_{(t,\theta_0)}^+ , \lambda^{(\Upsilon)}\circ \eps_{(t,\theta_0)}^+ \right)$ solves the same SDE for any $\theta_0 \leq \lambda^{(\Upsilon)}_t$, we get the result.
\end{proof}
This Lemma thereupon yields Proposition \ref{prop: non importance du saut}:
\begin{proposition}
    \label{prop: non importance du saut} 
    Let $F$ be an $\Fcal_\infty^H$-measurable random variable. Then for any $t\geq \Upsilon$ and for any $(\theta, \theta_0) \in \RR^2_+$,
    $$\ind{\theta \leq \lambda^{(\Upsilon)}_t} D_{(t,\theta)} \left( F \right) = \ind{\theta_0 \leq \lambda^{(\Upsilon)}_t} D_{(t,\theta_0)} \left( F \right)$$
\end{proposition}

\begin{notation}
    Let $t\geq \Upsilon$. For a $\Fcal_\infty^H$-measurable random variable $F$, we set
    $$D_t \left(F \right) := \ind{\theta = 0} D_{(t,\theta)}\left(F \right).$$
    In particular, for any $\theta \in \RR_+$, 
    $$\ind{\theta  \leq \lambda^{(\Upsilon)}_t} D_{(t,\theta)} \left( F\right) = \ind{\theta  \leq \lambda^{(\Upsilon)}_t} D_t \left( F \right).$$
\end{notation}

\begin{theorem}[Integration by parts]
\label{th:IPPH}
Set $\z:=(\z_{(t,\theta)})_{(t,\theta) \in (\Upsilon, +\infty)\times\RR_+}$ the stochastic process defined as 
$$ \z_{(t,\theta)}:= \mathds{1}_{\{\theta \leq \lambda^{(\Upsilon)}_t\}}, \quad (t,\theta) \in (\Upsilon, +\infty)\times\RR_+.$$
Let $Z:=(Z_t)_{t \geq \Upsilon}$ be a $\FF^H$-predictable process satisfying
$$\E{\int_\Upsilon^{+\infty} |Z_t|^2 \lambda^{(\Upsilon)}_t dt + \left(\int_\Upsilon^{+\infty} Z_t \lambda^{(\Upsilon)}_t dt \right)^2}<\infty.$$ It holds that 
\begin{itemize}
\item[(i)] $Z \z=(Z_t \mathds{1}_{\{\theta \leq \lambda^{(\Upsilon)}_t\}})_{(t,\theta)\in (\Upsilon, +\infty)\times\RR_+}$ belongs to $\Scal$.
\item[(ii)] For any $\Fcal^H_\infty$-measurable random variable $F$ with $\E{|F|^2}<+\infty$,
\begin{equation}
\label{eq:IBP}
\E{F \delta(Z \ind{\theta \leq \lambda^{(\Upsilon)}_t})} = \E{ \int_\Upsilon^{+\infty} \lambda^{(\Upsilon)}_t Z_t D_{t} F dt},
\end{equation}
where $D_{t} F = \mathds{1}_{\{0\leq \lambda^{(\Upsilon)}_t\}} D_{(t,0)} F $ .
\end{itemize}
\end{theorem}

The proofs of Proposition \ref{prop: non importance du saut} and Theorem \ref{th:IPPH} are not given in this paper. Indeed, the non linearity of the Hawkes does not change the proofs done in \cite{hillairet_malliavin-stein_2022}. That is why we refer to this article for detailed proofs.


\subsection{Convergence rate for continuous time Hawkes processes}

In this part, we set $(H,\lambda)$ a non-linear Hawkes process as defined in Definition \ref{def:Hawkes}, i.e. with intensity
    $$\lambda_t = h \left( \mu + \iint_{(0,t)\times \RR_+} \phi(t-s) \1_{\theta \leq \lambda_s} N(ds,d\theta) \right), \quad t\geq 0.$$
According to the previous work, Assumption \ref{assump: initiale} holds. Thanks to this we can write
    $$F^T := \frac{H_{T}- \int_0^{T} \lambda_u du}{\sqrt{T}} = \delta \left( \z^T Z^T \right)$$
    with $\z^T_{(t,\theta)} = \1_{\theta \leq \lambda_t}$ and $Z^T_t = \frac{\1_{[0,T]}(t)}{\sqrt{T}}$.
    This corresponds to the same quantity defined in the previous section where we have set:
    $$f^T = \left( \lambda_t \right)_{t\in [0,T]} \quad \text{and} \quad g^T \equiv 1.$$

Throughout this section, we will suppose that $\phi$ and $h$ respect Assumption \ref{assump: noyau}. Note that \cite{zhu_nonlinear_2013} establishes the convergence of $F^T$ to a Gaussian distribution in law under more restrictive assumption. Indeed, in \cite{zhu_nonlinear_2013}, $h$ is suppose to be non-decreasing and $\phi$ to be positive. In this article, we investigate the convergence rate of this process without supposing monotonicity on $h$ or sign for $\phi$ but only with Assumption \ref{assump: noyau}. We write our result in Theorem \ref{thm: conv_rate_hawkes_cont}.
    
\begin{theorem}
\label{thm: conv_rate_hawkes_cont}
    Let  $G \sim \Ncal (0,\sigma^2)$ where $\sigma^2 = \E{\lambda_0^\infty}$ .Then, under Assumption \ref{assump: noyau}, there exists $C>0$ such that for any $T>0$, we have
    $$d_W(F^T, G) \leq \frac{C}{\sqrt{T}}.$$
\end{theorem}
Before, giving the proof of Theorem \ref{thm: conv_rate_hawkes_cont}, let us write a remark about the state-of-the-art.

\begin{remark}
    Note that the constant depends on $ \alpha \Vert \phi \Vert_1 $, $ m := \int_0^t t\vert\phi(t)\vert \,dt $, and the expectation of a stationary Hawkes process. In particular, the condition $ \alpha \Vert \phi \Vert_1 <1 $ is crucial for the proof, as the constant may "explode" when $ \alpha \Vert \phi \Vert_1 \to 1 $.  
    Moreover, this condition is fundamental for the existence and uniqueness of the stationary version of the process, as established in \cite{bremaud_stability_1996}.
\end{remark}
\begin{remark}[Hawkes with inhibition]
    In this remark, we focus on Hawkes processes with inhibition. Those are non-linear Hawkes processes where $h$ is defined as the positive-part function, that is
    $$\lambda_t = \max \left(0,  \mu + \iint_{(0,t)\times \RR_+} \phi(t-s) \1_{\theta \leq \lambda_s} N(ds,d\theta) \right).$$
    Hawkes process with inhibition was first mentioned in \cite{bremaud_stability_1996} but the interest in such process occurs in the last decades. Law of Large Numbers, a Central Limit Theorem and large deviation results have been proved in \cite{cattiaux_limit_2022} where the authors restrict $\phi$ to be a compactly supported signed measurable function.\\
    We therefore extend their CLT result by considering $h: x \mapsto \max(0,x)$ and $\phi$ both satisfying Assumption \ref{assump: noyau}.
\end{remark}

\begin{remark}
    Malliavin's calculus is employed in \cite{torrisi_gaussian_2016} in a more general context but does not provide an optimal bound for the Wasserstein distance between a Gaussian distribution and the martingale $F^T$ defined by \eqref{def: F^T}. More recently, in \cite{hillairet_malliavin-stein_2022}, the bound was refined in the context of a linear Hawkes process with an exponential kernel. This was possible thanks to all the result of Subsection \ref{subsec: Malliavin} that is due to \cite{hillairet_malliavin-stein_2022}. In particular, they get
    $$d_W(F^T, G) \leq \frac{C}{\sqrt{T}}.$$
    Futhermore, a generalisation for linear Hawkes process has been developped in \cite{khabou_normal_2024}. This generalisation maintains the previous inequality for linear Hawkes processes with kernel with norm less than $1$. In our case, we move beyond the linearity assumption and prove a more general result.
\end{remark}

\begin{proof}[Proof of Theorem \ref{thm: conv_rate_hawkes_cont}]
    Let $(H,\lambda)$ be a Hawkes process as defined in Definition \ref{def:Hawkes}, i.e. with intensity
    $$\lambda_t = h \left( \mu + \iint_{(0,t)\times \RR_+} \phi(t-s) \1_{\rho \leq \lambda_s} N(ds,d\rho) \right), \quad t\geq 0.$$
    Since 
    $$F^T := \frac{H_{T}- \int_0^{T} \lambda_u du}{\sqrt{T}} = \delta \left( \z^T Z^T \right)$$
    with $\z^T_{(t,\theta)} = \1_{\theta \leq \lambda_t}$ and $Z^T_t = \frac{\1_{[0,T]}(t)}{\sqrt{T}}$, we can apply Theorem \ref{thm: main} to get 
    $$d_W(F, G) \leq \E{\left| \sigma^2 - \frac{1}{T}\int_0^T \lambda_u du \right|} + \frac{1}{T^{3/2}} \E{\int_0^T \lambda_u du} := T_1 + T_2 .$$
    Using the triangular inequality, we have
    $$T_1 \leq \E{\left| \frac{1}{T} \int_0^T \E{\lambda^{\infty}_t} dt - \frac{1}{T}\int_0^T \lambda_t^\infty dt\right|} + \frac{1}{T} \int_0^T \E{\left| \lambda_t^\infty - \lambda_t \right|} dt := T_{1,1} + T_{1,2}$$
    where $\sigma^2$ is the mean of the stationary Hawkes process defined by its intensity
    $$\lambda^{\infty}_t = h \left( \mu + \iint_{(-\infty,t)\times \RR_+} \phi(t-s) \1_{\rho \leq \lambda^{\infty}_s} N(ds,d\rho) \right) =: h \left( \mu + \xi^{\infty}_{t} \right), \quad t\geq 0.$$
    We split the proof into $3$ parts in which we deal with $T_{1,1}$, $T_{1,2}$ and $T_{2}$.\\
    
    \textbf{Upper bound of $T_{1,1}$.}\\
    First, let us use Cauchy-Schwarz inequality and get
    $$T_{1,1} \leq \sqrt{\frac{1}{T^2}  \E{ \left(\int_0^T \E{\lambda_t^\infty} -\lambda_t^\infty dt \right)^2}} = \sqrt{\frac{2}{T^2} \E{\int_0^T \int_0^t \left(\E{\lambda_t^\infty} -\lambda_t^\infty\right)\left(\E{\lambda_s^\infty} -\lambda_s^\infty\right)  ds dt}}.$$
    We fix $(t,s)\in [0,T]^2$. According to Clark-Ocone formula (see \cite{hillairet_poisson_2024}), 
    $$\lambda_t^\infty -  \E{\lambda_t^\infty} = \int_{\RR^2} \left( D_{(u,\rho)}\left(\lambda_t^\infty \right)\right)^\mathrsfs{P} \left(N(du,d\rho)- dud\rho \right) = \delta\left( \left(D_{(\cdot,\cdot)}\left(\lambda_t^\infty \right)\right)^\mathrsfs{P}\right).$$
    where $\left( D_{(u,\rho)}\left(\lambda_t^\infty \right)\right)^\mathrsfs{P}$ is the predictable projection of the Malliavin derivative of $\lambda_t^\infty$. Hence, using Lemma \ref{lem: prod_div}, we get
    $$\E{\left(\lambda_t^\infty -  \E{\lambda_t^\infty}\right)\left(\lambda_s^\infty -  \E{\lambda_s^\infty}\right)}  = \E{\int_0^s \int_{\RR_+} \left(D_{(u,\rho)}\left(\lambda_t^\infty \right)\right)^\mathrsfs{P} \left(D_{(u,\rho)}\left(\lambda_s^\infty \right) \right)^\mathrsfs{P} du d\rho }$$
    Moreover, for fixed $(u,\rho)\in \RR_+^2$, this predictable projection is almost surely equal to the conditional expectancy of the Malliavin derivative of $\lambda_t^\infty$ when it is conditioned by $\Fcal^N_{u^-}$ (see Remark 3.5 in \cite{zhang_clarkocone_2009}):
    $$\left( D_{(u,\rho)}\left(\lambda_t^\infty \right)\right)^\mathrsfs{P} = \EE_{u^-} \left[D_{(u,\rho)}\left(\lambda_t^\infty \right)\right], \quad a.s.$$

    Therefore, we have
    \begin{align}
    \E{\left(\lambda_t^\infty -  \E{\lambda_t^\infty}\right)\left(\lambda_s^\infty -  \E{\lambda_s^\infty}\right)} & = \E{\int_0^s \int_{\RR_+} \EE_u\left[D_{(u,\rho)}\left(\lambda_t^\infty \right)\right] \EE_u\left[D_{(u,\rho)}\left(\lambda_s^\infty \right) \right]du d\rho }
     \nonumber \\
    & \leq \E{\int_0^s \int_{\RR_+} \EE_u\left[ \left|D_{(u,\rho)}\left(\lambda_t^\infty \right)\right|\right] \EE_u\left[\left|D_{(u,\rho)}\left(\lambda_s^\infty \right) \right| \right]du d\rho }. \label{ineq prod delta}
    \end{align}
    Lemma \ref{lem: maj esp cond} yields to
    \begin{align*}
    \E{\left(\lambda_t^\infty -  \E{\lambda_t^\infty}\right)\left(\lambda_s^\infty -  \E{\lambda_s^\infty}\right)} &\leq \E{\int_0^s \int_{\RR_+} \1_{\rho \leq \lambda_u^{\infty}} \psi^{(\alpha)}(t-u) \psi^{(\alpha)}(s-u)} du d\rho \\
    &\leq \int_0^s \E{\lambda_u^{\infty} }\psi^{(\alpha)}(t-u) \psi^{(\alpha)}(s-u) du .
    \end{align*}
    Thus, we compute
    \begin{align*}
    T_{1,1} &\leq \sqrt{ \frac{2}{T^2} \int_0^T \int_0^t \int_0^s \E{\lambda_u^{\infty} }\psi^{(\alpha)}(t-u) \psi^{(\alpha)}(s-u) du ds dt}.
    \end{align*}
    We now make use of Fubini theorem to switch the integrals and to get
    \begin{align*}
    T_{1,1} &\leq \sqrt{ \frac{2}{T^2} \int_0^T \int_u^T \int_t^T \E{\lambda_u^{\infty} }\psi^{(\alpha)}(t-u) \psi^{(\alpha)}(s-u) ds dt du} \leq \sqrt{ \frac{2\E{\lambda_0^{\infty} } \left\Vert \psi^{(\alpha)} \right\Vert_1^2}{T^2}T} \leq \frac{C}{\sqrt{T}}.
    \end{align*}

    \textbf{Upper bound of $T_{1,2}$.}\\
    In this part, we find an upper bound of $T_{1,2}$ which we recall just below:
    $$T_{1,2} = \frac{1}{T} \int_0^T \E{\left| \lambda_t^\infty - \lambda_t \right|} dt.$$
    For $t\in [0,T]$, since $h$ is $\alpha$-Lipschitz, we get
    \begin{align*}
    &\left\vert \lambda_t^{\infty} - \lambda_t \right\vert \\
    &\leq \alpha \left\vert \int_{(-\infty,0] \times \RR_+} \phi(t-s) \1_{\rho \leq \lambda_s^{\infty}} N(ds,d\rho) + \int_{(0,t) \times \RR_+} \phi(t-s) \left( \1_{\rho \leq \lambda_s^{\infty}}- \1_{\rho \leq \lambda_s} \right)N(ds,d\rho)\right\vert .
    \end{align*}
    Therefore, 
    $$\E{\left\vert \lambda_t^{\infty} - \lambda_t \right\vert} \leq \alpha  \int_{-\infty}^t \left| \phi(t-s)\right| \E{\lambda^{\infty}_s} ds + \alpha \int_0^t \left|\phi(t-s)\right| \E{\left\vert \lambda_s^{\infty} - \lambda_s \right\vert}ds  .$$
    Moreover, using the fact that $\E{\lambda_{\cdot}^{\infty}}$ is constant, we obtain
    $$\E{\left\vert \lambda_t^{\infty} - \lambda_t \right\vert} \leq \alpha \E{\lambda^{\infty}_0} \int_{t}^{\infty} \left| \phi(s)\right| ds + \alpha\int_0^t \left|\phi(t-s)\right| \E{\left\vert \lambda_s^{\infty} - \lambda_s \right\vert}ds  .$$
    And so,
    \begin{align*}
        \int_0^T \E{\left\vert \lambda_t^{\infty} - \lambda_t \right\vert} dt &\leq \alpha \E{\lambda^{\infty}_0} \int_0^T \int_{t}^{\infty} \left| \phi(s)\right| ds dt+ \alpha \int_0^T \int_0^t \left|\phi(t-s)\right| \E{\left\vert \lambda_s^{\infty} - \lambda_s \right\vert}ds dt.
    \end{align*}
    We now make use of Fubini theorem and we get:
    \begin{align*}
        \int_0^T \E{\left\vert \lambda_t^{\infty} - \lambda_t \right\vert} dt  &\leq \alpha \E{\lambda^{\infty}_0} \int_0^T \int_{t}^{\infty} \left| \phi(s)\right| ds dt+ \alpha \int_0^T \left(\int_s^T \left|\phi(t-s)\right| dt \right) \E{\left\vert \lambda_s^{\infty} - \lambda_s \right\vert} ds\\
        & \leq \alpha \E{\lambda^{\infty}_0} \int_0^T \int_{t}^{\infty} \left| \phi(s)\right| ds dt+ \alpha \left\Vert \phi\right\Vert_1 \int_0^T \E{\left\vert \lambda_s^{\infty} - \lambda_s \right\vert} ds.
    \end{align*}
    
    Hence, by rearranging the terms, we get:
    \begin{equation}
    \label{eq: maj_esp_lambda_stat}
        \int_0^T \E{\left\vert \lambda_t^{\infty} - \lambda_t \right\vert} dt \leq \frac{\alpha \E{\lambda^{\infty}_0}}{1-\alpha \left\Vert \phi\right\Vert_1} \int_0^T \int_{t}^{\infty} \left| \phi(s)\right| ds dt \leq \frac{\alpha \E{\lambda^{\infty}_0}}{1-\alpha \left\Vert \phi\right\Vert_1} \int_0^{\infty} t \left| \phi(t)\right| dt = C.
    \end{equation}
    
    Note that the constant is well-defined thanks to the Assumption \ref{assump: noyau}. Therefore, we have
    $$T_{1,2} \leq \frac{C}{T}.$$
 
    \textbf{Upper bound of $T_{2}$}\\
    Since $T_2 = \frac{1}{T^{3/2}} \E{\int_0^T \lambda_t dt}$, we have:
    $$T_2 \leq \frac{1}{T^{3/2}} \int_0^T  \E{ \left| \lambda^{\infty}_t -\lambda_t \right| }dt + \frac{1}{T^{3/2}} \int_0^T  \E{\lambda^{\infty}_t } dt.$$
    Using the fact that $\E{\lambda_{\cdot}^{\infty}}$ is constant combining with \eqref{eq: maj_esp_lambda_stat}, we get
    $$T_2 \leq \frac{C}{T^{3/2}} + \frac{\E{\lambda^{\infty}_0}}{\sqrt{T}}.$$
\end{proof}

\begin{remark}
    After completing our proof, we realized that certain elements of our proof were similar to the proof of Theorem 6 in \cite{zhu_nonlinear_2013}. In particular, the use of a stationary version of a Hawkes process to circumvent the difficulties posed by a non-linear and non-stationary Hawkes process is employed. However, theses computations has not been used to prove quantification as we have done.
\end{remark}

\subsection{Convergence rate with variance reduction}
In this subsection, we consider $(H,\lambda)$ a Hawkes process as defined in Definition \ref{def:Hawkes}, i.e. with intensity
    $$\lambda_t = h \left( \mu + \int_{(0,t)\times \RR_+} \phi(t-s) \1_{\rho \leq \lambda_s} N(ds,d\rho) \right), \quad t\geq 0.$$
Whereas, Theorem \ref{thm: conv_rate_hawkes_cont} gives the convergence of non-linear Hawkes process towards a Gaussian in a more general context that the existing result in \cite{zhu_nonlinear_2013}, the variance of Gaussian is complicated to compute and it is mostly impossible to give an explicit value. To deal with this problem, we consider a "variance reduction" which will give a similar results but with the standard Gaussian $\Ncal (0,1)$. To do so, we define $F^T$ by taking $f^T = g^T = \lambda$, i.e.
$$F^T := \int_{[0,T]\times \RR_+} \1_{\theta \leq \lambda_t} \frac{1}{\sqrt{T \lambda_t}} \left( N(dt,d\theta)-dtd\theta\right) = \delta \left( \z Z^T \right)$$
with $\z_{(t,\theta)} :=  \1_{\theta \leq \lambda_t}$ and $Z_t^T := \frac{\1_{[0,T]}(t)}{\sqrt{T \lambda_t}}$ for  $T>0$ and $(t,\theta) \in [0,T]\times \RR_+$. Indeed, an easy computation yields $\text{Var}\left(\left|F^T\right|^2 \right)=1$ for all $T>0$.
\begin{theorem}
\label{thm: conv_rate_hawkes_cont2}
    Let  $G \sim \Ncal (0,1)$. Then, under Assumption \ref{assump: noyau} with $h$ non-decreasing and $\phi$ non-negative,
    $$d_W(F^T, G) \leq \frac{C}{\sqrt{T}}.$$
    Here, $C$ depends on $\alpha$ and $\Vert \phi \Vert_1$.
\end{theorem}

\begin{proof}
    We start this proof by make use of Theorem \ref{thm: main} in order to get
    \begin{align*}
        d_W(F^T,G) &\leq \E{\left( F^T\right)^3} +\frac{1}{T} \int_{0}^T \int_{t}^T\E{  \sqrt{\lambda_t}\sqrt{\lambda_s} \left| D_{(t,0)}\left( Z_s^T\right) \right| }ds dt \\
            & \quad + \frac{1}{\sqrt{T}} \int_{0}^T\int_{t}^T \E{  \sqrt{\lambda_t}\lambda_s \left|  \delta\left( \z D_{(s,0)}\left(Z^T\right)\right) D_{(t,0)}\left( Z_s^T\right)\right| } ds dt\\
            &\quad -\frac{2}{T}\E{\int_{0}^T  \sqrt{\lambda_t} \EE_t\left[\int_t^T \sqrt{\lambda_s} D_{(t,0)}\left(Z^T_s \right) ds \right] dt . }
    \end{align*}
    Thereupon, thanks to Lemma \ref{lem: maj_Dlambda12}, we have
    \begin{align*}
        d_W(F^T,G) &\leq \E{\left( F^T\right)^3} +\frac{1}{T^{3/2}} \int_{0}^T \int_{t}^T\E{  \sqrt{\lambda_t}\frac{D_{(t,0)}\left( \lambda_s\right)}{\lambda_s} }ds dt \\
            & \quad + \frac{1}{T} \int_{0}^T\int_{t}^T \E{  \sqrt{\lambda_t} \left|  \delta\left( \z D_{(s,0)}\left(Z^T\right)\right) \right| \frac{D_{(t,0)}\left( \lambda_s \right)}{\sqrt{\lambda_s}}} ds dt\\
            &\quad +\frac{2}{T^{3/2}}\E{\int_{0}^T  \sqrt{\lambda_t} \EE_t\left[\int_t^T \frac{D_{(t,0)}\left( \lambda_s\right)}{\lambda_s} ds \right] dt .}
    \end{align*}
    Now, since $h$ is non-decreasing and $\phi$ is a non-negative function, we have $\lambda_s \geq h(\mu),\, s \in [0,T]$ and so
    \begin{align*}
        d_W(F^T,G) &\leq \E{\left( F^T\right)^3} +\frac{C}{T^{3/2}} \int_{0}^T \int_{t}^T\E{  \sqrt{\lambda_t}D_{(t,0)}\left( \lambda_s\right) }ds dt \\
            & \quad + \frac{C}{T} \int_{0}^T\int_{t}^T \E{  \sqrt{\lambda_t} \left|  \delta\left( \z D_{(s,0)}\left(Z^T\right)\right) \right| D_{(t,0)}\left( \lambda_s \right)} ds dt\\
            &\quad +\frac{C}{T^{3/2}}\E{\int_{0}^T  \sqrt{\lambda_t} \EE_t\left[\int_t^T D_{(t,0)}\left( \lambda_s\right) ds \right] dt }.
    \end{align*}
    In addition, using Lemma \ref{lem: maj esp cond}, we obtain that
    \begin{align}
        d_W(F^T,G) &\leq \E{\left( F^T\right)^3} +\frac{C}{T^{3/2}} \int_{0}^T \E{  \sqrt{\lambda_t} \int_{t}^T\psi^{(\alpha)}(t-s)ds} dt \nonumber \\
            & \quad + \frac{C}{T} \int_{0}^T\int_{t}^T \E{  \sqrt{\lambda_t} \left|  \delta\left( \z D_{(s,0)}\left(Z^T\right)\right) \right| D_{(t,0)}\left( \lambda_s \right)} ds dt \label{eq: Maj_dist_proof_var_reduc}\\
            &\quad +\frac{C}{T^{3/2}}\E{\int_{0}^T  \sqrt{\lambda_t} \int_t^T \psi^{(\alpha)}(t-s) ds dt }. \nonumber
    \end{align}
    We deal with the term containing the divergence operator by employing Lemma \ref{lemma: maj_esp_cond_delta} and the inequality $\lambda \geq h(\mu)$. Such computation yields
    \begin{align*}
        \EE_s\left[|\delta\left(\z^T D_{(s,0)} Z^T)\right)|\right] &\leq \frac{C}{\sqrt{T}} \EE_s \left[ \int_{s}^T \frac{D_{(s,0)}\left(\lambda_u \right)}{\sqrt{\lambda_u}}  du \right] \leq \frac{C}{\sqrt{T}} \int_{s}^T \EE_s \left[ D_{(s,0)}\left(\lambda_u \right) \right] du .
    \end{align*}
    Following on from this, we make use of Lemma \ref{lem: maj esp cond} followed by the inequality $\Vert \psi^{(\alpha)} \Vert_1 \leq \frac{1}{1-\alpha \Vert \phi \Vert_1} <+\infty$ and we get
    \begin{align*}
        \EE_s\left[|\delta\left(\z^T D_{(s,0)} Z^T)\right)|\right] &\leq \frac{C}{\sqrt{T}} \int_{s}^T \psi^{(\alpha)}(u-s)  du \leq \frac{C}{\sqrt{T}}.
    \end{align*}
    Hence, implementing this in Equation \eqref{eq: Maj_dist_proof_var_reduc}, we have 
    \begin{align*}
        d_W(F^T,G) \leq & \E{\left( F^T\right)^3} +\frac{C}{T^{3/2}} \int_{0}^T \E{  \sqrt{\lambda_t} \int_{t}^T\psi^{(\alpha)}(t-s)ds} dt \\
        &+ \frac{C}{T^{3/2}} \int_{0}^T \int_{t}^T \E{  \sqrt{\lambda_t}  D_{(t,0)}\left( \lambda_s \right)} ds dt .
    \end{align*}
     In the same way as before, we employ Lemma Lemma \ref{lem: maj esp cond} and the inequality $\Vert \psi^{(\alpha)} \Vert_1 \leq \frac{1}{1-\alpha \Vert \phi \Vert_1} <+\infty$ to get
   $$d_W(F^T,G) \leq \E{\left( F^T\right)^3} +\frac{C}{T^{3/2}} \int_{0}^T \E{  \sqrt{\lambda_t}} dt .$$
    Finally, since $\E{\left( F^T\right)^3} \leq \frac{C}{\sqrt{T}}$ (see Lemma \ref{lem: moment 3}) and since $\E{\lambda_t} \leq C$, we obtain
    $$d_W(F^T ,G) \leq \frac{C}{\sqrt{T}} + \frac{C}{T^{3/2}} \int_{0}^T \sqrt{\E{  \lambda_t}} dt \leq  \frac{C}{\sqrt{T}}.$$
\end{proof}


\section{Application to Locally stationary Hawkes processes}
\label{sec:locally_stat}
This section relies mostly on the work of \cite{deschatre_limit_2025}. In their article, the author prove a LLN and fCLT for a class of locally stationary Hawkes processes. These processes are generalization of linear Hawkes process. Indeed, in the context of such processes, the baseline and the kernel of the intensity $\left(\lambda^T_t\right)_{t\in [0,T]}$ are both time varying and depend on $\frac{t}{T}$. Whereas these processes was defined in \cite{roueff_locally_2016} by Roueff et al., we do not exactly fit with their definition and decide to consider a special case that has been studied in \cite{deschatre_limit_2025}. In particular, we define the process as follow:
\begin{definition}
    A locally stationary Hawkes process on $[0,T]$ is any point process $H^T$ of the form \eqref{eq: point process} with intensity
    $$\lambda^T_t = \mu\left( \frac{t}{T}\right) + \int_{(0,t)} \gamma \left( \frac{t}{T}\right) \phi(t-s) dH^T_s, \quad t\in [0,T],$$
    where $\mu : [0,1] \to \RR_+$, $\gamma: [0,1] \to \RR_+$ and $\phi: \RR_+ \to \RR_+$.
\end{definition}
\begin{assumption}
\label{assump: locally stationary}
    The baseline and reproduction function $\mu$ and $\gamma$ are respectively continuous and continuously differentiable over $[0,1]$. Moreover, $\phi$ is integrable and such that 
    $$\left\| \phi \right\|_1 \sup_{x\in [0,1]} \left| \gamma(x)\right| <1.$$
\end{assumption}
We also recall Theorem 2 of \cite{deschatre_limit_2025}.
\begin{theorem}
    Under Assumption \ref{assump: locally stationary}, 
    $$\frac{1}{\sqrt{T}} \left( N^T_{Tu} - \E{N^T_{Tu}}\right) \xrightarrow[T \to +\infty]{} \int_0^u \left(1- \gamma(s) \left\| \phi \right\|_1 \right)^{-3/2} \mu(s)^{1/2} dB_s$$
    where $B$ stands for the standard Brownian motion.
\end{theorem}
One can remark that if $\gamma$ is a constant function, the latter Theorem is similar to existing ones for linear Hawkes processes. We now define $F^T$ as previously with $f^T = \lambda^T$ and $g^T \equiv 1$, that is:
$$F^T := \frac{N^T_T - \int_0^T \lambda^T_s ds}{\sqrt{T}}=\frac{1}{\sqrt{T}}\iint_{(0,T]\times \RR_+} \1_{\theta \leq \lambda^T_s} \Ntil (ds,d\theta).$$
With these notations, we obtain the following theorem:
\begin{theorem}
    Suppose that Assumption \ref{assump: locally stationary} is in order and suppose that $\mu$ is Lipschitz and for any $t \in \RR_+$,$\int_0^t \phi(s) \sqrt{s} ds <+\infty$. Then, there exists $C>0$ such that for any $T>0$, 
    $$d_W \left( F^T , G\right) \leq \frac{C}{\sqrt{T}}$$
    where $G \sim \Ncal \left(0, \sigma^2 \right)$ with $\sigma^2 = \int_0^1 \left(1-\gamma(x)\left\|\phi \right\|_1 \right)^{-1}\mu(x) dx$.
\end{theorem}
\begin{proof}
    By applying Theorem \ref{thm : Maj_g_determinist}, we obtain that there exists $C>0$ such that for any $T>0$, 
    $$d_W \left( F^T , G\right) \leq \E{\left|\sigma^2 - \frac{1}{T} \int_0^T \lambda^T_s ds \right|} + \frac{C}{\sqrt{T}}.$$
    Moreover, by the triangular inequality,
    $$\E{\left|\sigma^2 - \frac{1}{T} \int_0^T \lambda^T_s ds \right|} \leq \left|\sigma^2 - \frac{1}{T} \int_0^T \E{\lambda^T_s} ds \right| + \E{\left|\frac{1}{T} \int_0^T \E{\lambda^T_s}-\lambda^T_s ds \right|}.$$
    According to Lemmas 13 and 11 of \cite{deschatre_limit_2025}, there exists $C>0$ such that for any $T>0$,
    $$\left|\frac{1}{T} \int_0^T \E{\lambda^T_s}-\lambda^T_s ds \right|^2 \leq \frac{C}{T}.$$
    Besides, since $\mu$ is supposed to be Lipschitz and for any $t \in \RR_+$,$\int_0^t \phi(s) \sqrt{s} ds <+\infty$, Proposition 3 of \cite{deschatre_limit_2025} gives:
    $$\left| \frac{1}{T} \int_0^T \int_s^T \Gamma^T(t,s) \mu \left(\frac{s}{T}\right) dt ds - \int_0^1 \left[\left(1-\gamma(x)\left\|\phi \right\|_1 \right)^{-1}-1\right]\mu(x) dx\right| = o\left( \frac{1}{\sqrt{T}}\right)$$
    which yields
    $$\left|\sigma^2 - \frac{1}{T} \int_0^T \E{\lambda^T_s} ds \right| = o\left( \frac{1}{\sqrt{T}}\right).$$
\end{proof}

 
\section{Application to Discrete Hawkes processes}
\label{sec:discrete}
In this section, we leverage our main result to examine a specific instance of the Hawkes process: the discrete Hawkes process. This class of Hawkes processes has been previously investigated by Seol in \cite{seol_limit_2015}. In their work, they establish that when considering a Hawkes process founded on the Bernoulli distribution and subject to certain conditions, it converges in distribution to a Gaussian distribution. Building upon this foundation, \cite{quayle_etude_2022} extended this analysis to Hawkes processes based on the Poisson distribution. Here, we employ Theorem \ref{thm : Maj_g_determinist} to determine the convergence rate of the aforementioned limit theorem for a Hawkes process governed by the Poisson law. Let us define such Hawkes processes.

\begin{definition}
    Let $(\alpha_k)_{k\in \NN}$ be a sequence of positive real numbers such that
    $$\sum\limits_{k=0}^n \alpha_k <1\quad \text{and} \quad \sum\limits_{k=0}^n k\alpha_k < + \infty .$$
    We define $\left(X_k \right)_{k\in \NN^*}$ a sequence of random variables such that $X_1 \sim\Pcal(\alpha_0)$ and for any $k\in \NN^*$, $X_{k} \mid \Fcal^{(d)}_{k-1} \sim \Pcal \left( \alpha_0 + \sum\limits_{i=1}^{k-1} \alpha_{k-i} X_i \right)$ where $\Fcal^{(d)}_{k-1} = \sigma \left( X_i,\, i\leq k-1 \right)$ and $\Pcal(\beta)$ stands for the Poisson law with mean $\beta >0$.\\
    A discrete Hawkes process $(H_n)_{n\in\NN}$ is given by $H_n = \sum\limits_{k=1}^n X_k$, $n\in \NN$. 
\end{definition}

With this definition, we can introduce the following notations:
\begin{align}
    \lambda^{(d)}_k &:= \alpha_0 + \sum\limits_{i=1}^{k-1} \alpha_{k-i} X_i, \quad k\in \NN ;\\
    \lambda_t &:= \alpha_0 \1_{[0,1)}(t) + \sum\limits_{k=2}^{\infty} \lambda^{(d)}_k \1_{[k-1,k)}(t), \quad t\in \RR_+.
\end{align}

Moreover, since $X_k = \int_{[k-1,k)\times \RR_+} \1_{\theta \leq \lambda_t} N(dt,d\theta)$ in law, we get 
$$H_n = \int_{(0,n]\times \RR_+} \ind{\theta \leq \lambda_t} N(dt,d\theta), \quad n\in \NN.$$
We consider the normalized martingale $(F^n)_{n\in \NN}$ define by
$$F^n := \frac{H_n - \int_0^n \lambda_t dt}{\sqrt{n}} = \int_{(0,n]\times \RR_+} \frac{\ind{\theta \leq \lambda_t}}{\sqrt{n}} \left( N(dt,d\theta) -dtd\theta \right), \quad n\in \NN.$$
Note that by construction, it is evident that $\lambda$ is $\FF$-predictable. Let us introduce two useful notations that will be used throughout this section.
\begin{notation}
    For the sequence $\left(\alpha_k\right)_{k\in \NN}$, we define $|\alpha |$ and $\varsigma^2$ by
    $$| \alpha | := \sum\limits_{k\geq 1} \alpha_k \quad \text{and} \quad \varsigma^2 := \frac{\alpha_0}{1-|\alpha |}.$$
\end{notation}
In order to apply Theorem \ref{thm: main}, we need first to verify the regularity assumption reacquired by it. We then give some results that can be found in \cite{quayle_etude_2022}.
\begin{proposition}
\label{prop: maj_moment_Xk}
    For any $k\geq 2$, 
    $$\sup\limits_{n \in \NN} \E{X_n} \leq \frac{\alpha_0}{1-|\alpha|}=\varsigma^2 \quad \text{and}\quad \sup\limits_{n \in \NN} \E{X_n^k} < +\infty.$$
\end{proposition}

\begin{proposition}[\cite{quayle_etude_2022}, Proposition 1.15]
\label{prop: reecriture_de_F_n}
    For any $n\geq 1$,
    $$\sqrt{n}F^n = \left( 1 -| \alpha | \right)H_n - n\alpha_0 + \eps_n,$$
    where $\left( \eps_n\right)_{n\in \NN}$ is a sequence of non-negative random variables such that:
    $$\frac{\eps_n}{\sqrt{n}} \xrightarrow[n\to +\infty]{\PP} 0\quad \text{and} \quad \E{\eps_n} \leq \left(1+ \sum\limits_{k\geq 1} k \alpha_k \right) \sup_{n\geq 1} \E{X_n}.$$
\end{proposition}

Thanks to Proposition \ref{prop: maj_moment_Xk}, we can prove the following lemmas. Since the lack of dependency of $n$ in Lemma \ref{lemma: discret_verif_hyp_2} is not require, we decide to not write the proof of the first inequality. Yet, we have detailed
the proof of the second one.

\begin{lemma}
\label{lemma: discret_verif_hyp_2}
    For any $n\in \NN$, $$\E{\left( \int_0^n \lambda_t dt\right)^2} < +\infty \quad \text{and} \quad \sup\limits_{n\in \NN} \E{\frac{1}{n} \int_0^n \lambda_t dt } \leq \varsigma^2.$$
\end{lemma}

\begin{proof}[Proof of Lemma \ref{lemma: discret_verif_hyp_2}]
     For any $k\in \NN$, Proposition \ref{prop: maj_moment_Xk} yields $ \E{X_k} \leq \varsigma^2$. Hence,
    $$\E{\lambda_k^{(d)}} = \alpha_0 + \sum_{i=1}^{k-1} \alpha_{k-i} \E{X_i} \leq \alpha_0 + \varsigma^2 \sum_{i=1}^{k-1} \alpha_{i}  
 = \varsigma^2.$$
    Therefore, for any $n\in \NN$,
    $$\E{\frac{1}{n}\int_0^n \lambda_t dt}  \leq \frac{\alpha_0}{1-|\alpha |}.$$
    This inequality holds for any $n \in \NN$ so it is true for the supremum over $n$.
\end{proof}

\begin{theorem}
\label{thm : discret_martingale}
    There exists $C>0$ such that for $n\in \NN$,
    $$d_W( F^n, G) \leq \frac{C}{\sqrt{n}}, \quad G\sim \Ncal(0,\varsigma^2).$$
\end{theorem}

\begin{proof}
   For $f:= \left(\lambda_t^n \right)_{t\in [0,n]}$ and $g \equiv 1$, we have that $f\in \Ccal$, $g\in \Ccal$. Moreover, thanks to Lemma \ref{lemma: discret_verif_hyp_2}, $f$ and $g$ verify all the assumptions needed to apply Theorem \ref{thm : Maj_g_determinist}. Then, we have that there exists $C>0$ such that for any $n\in \NN$,
   \begin{equation}
   \label{eq: maj_discret}
       d_W(F^n, G) \leq \E{\left|\varsigma^2 - \frac{1}{n} \int_0^n \lambda_t dt \right|} + \frac{C}{\sqrt{n}}, \quad G \sim \Ncal \left(0, \varsigma^2 \right).
   \end{equation}
    We now deal with the first right-hand side term of the inequality. Fix $n\in \NN$. Lemma (1.17) in \cite{quayle_etude_2022} yields
    $$\varsigma^2 - \frac{1}{n}\int_0^n \lambda_t dt = \varsigma^2 - \langle F,F\rangle_n = \frac{|\alpha |}{1-|\alpha |} \left( \alpha_0 - \left(1-|\alpha | \right)\frac{H_n}{n} -  \frac{1-|\alpha |}{|\alpha |} \frac{\eps_n}{n}\right)$$
    where $\langle \cdot, \cdot\rangle$ stands for the quadratic variation and $\eps_n$ is the same as in Proposition \ref{prop: reecriture_de_F_n}.
    Using Proposition \ref{prop: reecriture_de_F_n}, we have:
    \begin{align*}
        \varsigma^2 - \frac{1}{n}\int_0^n \lambda_t dt &=  \frac{|\alpha |}{1-|\alpha |} \left( \frac{\eps_n}{n} - \frac{F^n}{\sqrt{n}} - \frac{1-|\alpha |}{|\alpha |} \frac{\eps_n}{n}\right) =\frac{2| \alpha |-1}{1- |\alpha|} \frac{\eps_n}{n} - \frac{|\alpha |}{1-|\alpha |} \frac{F^n}{\sqrt{n}}.
    \end{align*}
    So, the Triangular Inequality with the positivity of $\eps_n$ lead to
    \begin{align*}
        \E{\left|\varsigma^2 - \frac{1}{n} \int_0^n \lambda_t dt \right|} &\leq \frac{\left|1-2| \alpha |\right|}{1- |\alpha|} \E{\frac{\eps_n}{n}} + \frac{|\alpha |}{1-|\alpha |} \E{\frac{\left|F^n\right|}{\sqrt{n}}}\\
        & \leq \frac{C}{n} + \frac{C}{\sqrt{n}}\sqrt{\E{\left|F^n\right|^2}}.
    \end{align*}
    Finally, by noticing that 
    $$\sup\limits_{n\in \NN}\E{\left| F^n\right|^2} = \sup\limits_{n\in \NN} \E{\frac{1}{n}\int_0^n \lambda_t dt} \leq \overline{C},$$
    we get the result from \eqref{eq: maj_discret}.
\end{proof}

\begin{theorem}
    There exists $C>0$ such that for any $n\in \NN$,
    $$d_W\left(\frac{H_n-n\varsigma^2}{\sqrt{n}}, \overline{G} \right)\leq \frac{C}{\sqrt{n}},\quad \overline{G} \sim \Ncal\left(0, \frac{\varsigma^2}{\left(1-| \alpha | \right)^2} \right) . $$
\end{theorem}

\begin{proof}
    First we write $\overline{F}^n:= \frac{H_n-n\varsigma^2}{\sqrt{n}}$. By definition of the Wasserstein distance, 
    $$d_W\left(\frac{H_n-n\varsigma^2}{\sqrt{n}}, \overline{G} \right) = \sup\limits_{u\in \mathrsfs{L}_1} \left\vert \E{u\left( \overline{F}^n \right)} - \E{u \left( \overline{G}\right)}\right\vert.$$
    Moreover, Proposition \ref{prop: reecriture_de_F_n} gives
    $$ \overline{F}^n = \frac{1}{1-| \alpha |} \left( F^n - \frac{\eps_n}{\sqrt{n}} \right)$$
    with $F^n$ and $\eps_n$ as defined in the previous theorem. Therefore, we have
    \begin{align*}
        d_W\left(\frac{H_n-n\varsigma^2}{\sqrt{n}}, \overline{G} \right) &\leq \sup\limits_{u\in \mathrsfs{L}_1} \left\vert \E{u\left( \frac{1}{1-| \alpha |} \left( F^n - \frac{\eps_n}{\sqrt{n}} \right) \right)} - \E{u\left( \frac{1}{1-| \alpha |} F^n\right)}\right\vert \\
        &\quad + \sup\limits_{u\in \mathrsfs{L}_1} \left\vert \E{u\left( \frac{1}{1-| \alpha |}  F^n \right)} - \E{u\left( \overline{G} \right)}\right\vert\\
        & = \sup\limits_{u\in \mathrsfs{L}_1} \left\vert \E{u\left( \frac{1}{1-| \alpha |} \left( F^n - \frac{\eps_n}{\sqrt{n}} \right) \right)} - \E{u\left( \frac{1}{1-| \alpha |} F^n\right)}\right\vert \\
        &\quad + d_W \left(\frac{1}{1-| \alpha |}  F^n,  \frac{1}{1-| \alpha |}  G\right)
    \end{align*}
    where $G \sim \Ncal \left(0,\varsigma^2 \right)$.\\
    Thanks to Theorem \ref{thm : discret_martingale}, we have that
    $$d_W \left(\frac{1}{1-| \alpha |}  F^n,  \frac{1}{1-| \alpha |}  G\right) = \frac{1}{1-| \alpha |} d_W(F^n,G) \leq \frac{C}{\sqrt{n}}.$$
    Let us deal with the other term. Take $u\in \mathrsfs{F}_W$ (defined in Equation \eqref{eq: def_FW}), we have:
    $$\left\vert \E{u\left( \frac{1}{1-| \alpha |} \left( F^n - \frac{\eps_n}{\sqrt{n}} \right) \right)} - \E{u\left( \frac{1}{1-| \alpha |} F^n\right)}\right\vert \leq \frac{1}{1-| \alpha |}\E{ \frac{\eps_n}{\sqrt{n}}}.$$
    Since this inequality holds for any $u\in \mathrsfs{F}_W$, it is also true when we take the supremum, i.e.
    $$\sup\limits_{u\in \mathrsfs{L}_1} \left\vert \E{u\left( \frac{1}{1-| \alpha |} \left( F^n - \frac{\eps_n}{\sqrt{n}} \right) \right)} - \E{u\left( \frac{1}{1-| \alpha |} F^n\right)}\right\vert \leq \frac{1}{1-| \alpha |}\E{ \frac{\eps_n}{\sqrt{n}}}.$$
    As for the proof of Theorem \ref{thm : discret_martingale}, the sequence $\left(\E{ \eps_n} \right)_{n\in \NN}$ is bounded. Therefore, there exists $C>0$ such that for any $n\in \NN$,
    $$d_W\left(\frac{H_n-n\varsigma^2}{\sqrt{n}}, \overline{G} \right)\leq \frac{C}{\sqrt{n}}.$$
\end{proof}


\section{Application to Nearly Unstable Linear Hawkes Processes}
\label{sec:nearly}

In this part, we give an application of Theorem \ref{thm: main}. In particular, we take $f= g= \left(\lambda_t^T \right)_{t\in [0,T]}$ where $\lambda^T$ is the intensity of a nearly unstable Hawkes process observed on $[0,T]$ (see \eqref{eq: int_nearly}). This Hawkes process has been introduce in \cite{jaisson_limit_2015}. We adopt their notations, which we outline below. 

\begin{notation}
    We consider $(a_T)_{T>0}$ a family of positive real numbers with $a_T <1$ such that : 
    $$ T (1-a_T) =1.$$
    Moreover, for $T>0$, we denote by $\phi^T$ a function such that $\phi^T = a_T \phi $ where $\phi$ is a non-negative function satisfying:
    $$\int_0^{+\infty} \left|\phi(t) \right| dt = 1 \quad \text{and} \quad \int_0^{+\infty} t\phi(t) dt = m <+\infty .$$
\end{notation}
With this notation, we introduce the nearly Hawkes process $H^T = \left( H^T_t\right)_{t\geq 0}$ as a linear empty history Hawkes process  where the intensity is given by 
\begin{equation}
\label{eq: int_nearly}
    \lambda^T_t := \mu + \iint_{(0,t)\times \RR} \phi^T(t-s) \ind{\theta \leq \lambda^T_s} N(ds,d\theta), \quad t\in [0,T].
\end{equation}

We decide to apply our results to two different normalized martingales. In the first case, we consider $f= g = \left(\E{\lambda^T_t} \right)_{t\in [0,T]}$ (see Proposition \ref{prop: nearly_esperance})  whereas in the second one, we study the case $f= g = \left(\lambda^T_t \right)_{t\in [0,T]}$. 

\begin{notation}
    For $T>0$, we denote by $\hat{F^T}$ and by $F^T$ the following quantities:
    $$F^T = \iint_{[0,T] \times \RR_+} \frac{\ind{\theta \leq \lambda_t^T}}{\sqrt{T \lambda_t^T}} \left(N(dt,d\theta) -dt d\theta \right) \quad \text{and} \quad \hat{F^T} = \iint_{[0,T] \times \RR_+} \frac{\ind{\theta \leq \E{\lambda_t^T}}}{\sqrt{T \E{\lambda_t^T}}} \left(N(dt,d\theta) -dt d\theta \right).$$
\end{notation}
As for the previous applications, Assumptions \ref{assump: dans S} and \ref{assump: initiale} hold in both cases. 

\paragraph{First case: $f= g = \left(\E{\lambda^T_t} \right)_{t\in [0,T]}$}~\\

In this proposition, we apply our Theorem \ref{thm : Maj_g_determinist} to $\hat{F^T}$.

\begin{proposition}
\label{prop: nearly_esperance}
    There exists $C>0$ such that for any $T>0$,
    $$d_W(\hat{F^T},G) \leq \frac{C}{\sqrt{T}}, \quad G \sim \Ncal \left(0,1 \right)$$
    where $C$ depends on $\Vert \phi \Vert_1$, $\int_0^{+\infty} t \phi(t) dt$ and $\mu$.
\end{proposition}
\begin{proof}
    Apply Theorem \ref{thm : Maj_g_determinist}.
\end{proof}

\paragraph{Second case: $f= g = \left(\lambda^T_t \right)_{t\in [0,T]}$}~\\

\begin{proposition}
    There exists $C>0$ such that for all $T>0$, we have
    \begin{align*}
        d_W(F^T,G) &\leq \frac{C}{\sqrt{T}} + \frac{1}{T}\E{ \int_{0}^T \sqrt{\lambda_t^T} \sqrt{\EE_t \left[\int_t^T \left( 1 - \sqrt{\frac{\lambda_s^T}{\lambda_s^T\circ \eps_{(t,0)}^{+}}}\right)^2 ds\right]}dt} \\
        & \quad +\frac{1}{T^{3/2}}\E{ \int_{0}^T \sqrt{\lambda_t^T} \EE_t \left[\int_t^T \left( 1 - \sqrt{\frac{\lambda_s^T}{\lambda_s^T\circ \eps_{(t,0)}^{+}}}\right)^2 ds\right] dt}\\
        & \quad + \frac{2}{T^{3/2}}\E{ \int_{0}^T \sqrt{\int_t^T \left( 1 - \sqrt{\frac{\lambda_s^T}{\lambda_s^T\circ \eps_{(t,0)}^{+}}}\right)^2 ds }dt}.
    \end{align*}
\end{proposition}

\begin{proof}
    By applying Theorem \ref{thm: main}, we have that there exists $C>0$ such that for $T>0$,
    \begin{align*}
        d_W(F^T,G) &\leq \frac{C}{\sqrt{T}} + \frac{1}{\sqrt{T}}\E{ \int_{0}^T \sqrt{\lambda_t^T}\left|\delta \left( \z^T D_{(t,0)}\left(Z^T \right)\right) \right| dt} \\
            & \quad +\frac{1}{\sqrt{T}}\E{ \int_{0}^T \sqrt{\lambda_t^T} \left|\delta \left( \z^T D_{(t,0)}\left(Z^T \right)\right) \right|^2 dt}\\
            & \quad + \frac{2}{T}\E{ \int_{0}^T \left|\delta \left( \z^T D_{(t,0)}\left(Z^T \right)\right) \right|dt}.
    \end{align*}
    Moreover, for $t\in [0,T]$,
    \begin{align*}
    \EE_t \left[ \left|\delta \left( \z^T D_{(t,0)}\left(Z^T \right)\right) \right|^2 \right] &= \frac{1}{T}\EE_t \left[ \int_t^T  \lambda_s^T \left| D_{(t,0)}\left(\frac{1}{\sqrt{\lambda_s^T}} \right) \right|^2 ds \right] \\
    &= \frac{1}{T}\EE_t \left[\int_t^T \left( 1 - \sqrt{\frac{\lambda_s^T}{\lambda_s^T\circ \eps_{(t,0)}^{+}}}\right)^2 ds\right]. 
    \end{align*}
    Thus, we have:
        \begin{align*}
        d_W(F^T,G) &\leq \frac{C}{\sqrt{T}} + \frac{1}{T}\E{ \int_{0}^T \sqrt{\lambda_t^T} \sqrt{\EE_t \left[\int_t^T \left( 1 - \sqrt{\frac{\lambda_s^T}{\lambda_s^T\circ \eps_{(t,0)}^{+}}}\right)^2 ds\right]}dt} \\
            & \quad +\frac{1}{T^{3/2}}\E{ \int_{0}^T \sqrt{\lambda_t^T} \EE_t \left[\int_t^T \left( 1 - \sqrt{\frac{\lambda_s^T}{\lambda_s^T\circ \eps_{(t,0)}^{+}}}\right)^2 ds\right] dt}\\
            & \quad + \frac{2}{T^{3/2}}\E{ \int_{0}^T \sqrt{\int_t^T \left( 1 - \sqrt{\frac{\lambda_s^T}{\lambda_s^T\circ \eps_{(t,0)}^{+}}}\right)^2 ds }dt}.
    \end{align*}
\end{proof}

We propose in the following another upper bound of distance. This second upper bound allows us to link the distance with the third moment of $F^T$.

\begin{proposition}
    For $G\sim \Ncal(0,1)$, there exists $C>0$ such that for all $T>0$, we have
    \begin{align*}
    d_W(F^T, G) &  \leq \frac{C}{\sqrt{T}} +\frac{C}{T}\E{\int_{0}^T  \sqrt{\lambda_t^T} \sqrt{ \int_t^T \EE_t\left[\left| 1-\frac{\lambda^T_s}{\lambda^T_s \circ \eps_{(t,0)}^+} \right|^2 \right]ds }  dt } .
\end{align*}
\end{proposition}

\begin{proof}
    Using the third part of Theorem \ref{thm: main}, we get
    \begin{align*}
        d_W\left(F^T,G \right) & \leq  \E{\left( F^T\right)^3} + \frac{1}{\sqrt{T}} \E{\int_{0}^T  \sqrt{\lambda_t^T}\left|\delta \left(\z^T D_{(t,0)}\left(  Z^T\right)\right)\right| dt }\\
        &\quad +\frac{2}{T}\E{\int_{0}^T  \sqrt{\lambda_t^T} \EE_t\left[\int_t^T \sqrt{\lambda^T_s} \left|D_{(t,0)}\left(Z^T_s \right)\right| ds \right] dt }.\\
    \end{align*}
    Besides, thanks to Lemma \ref{lem: moment 3}, we have
    $$\E{\left( F^T \right)^3} = \frac{1}{T^{3/2}} \int_0^T \E{\frac{1}{\sqrt{\lambda_t^T}}} dt \leq \frac{1/\mu}{\sqrt{T}}.$$
    Thus, we have that for any $v\in \mathrsfs{F}_W$,
    \begin{align*}
        \E { F^T v(F^T)-v'\left( F^T\right)} &\leq  \frac{C}{\sqrt{T}} + \frac{1}{\sqrt{T}}\E{\int_{0}^T  \sqrt{\lambda_t^T}\left|\delta \left(\z^T D_{(t,0)}\left(  Z^T\right)\right)\right| dt }\\
        &\quad +\frac{2}{T}\E{\int_{0}^T  \sqrt{\lambda_t^T} \EE_t\left[\int_t^T \sqrt{\lambda^T_s} \left| D_{(t,0)}\left(Z^T_s \right) \right| ds \right] dt }.
    \end{align*}
    For $G\sim \Ncal \left(0,1\right)$, we get
    \begin{align*}
        d_W \left( F^T , G\right) &\leq  \frac{C}{\sqrt{T}} + \frac{1}{\sqrt{T}}\E{\int_{0}^T  \sqrt{\lambda_t^T}\left|\delta \left(\z^T D_{(t,0)}\left(  Z^T\right)\right)\right| dt }\\
        &\quad +\frac{2}{T}\E{\int_{0}^T  \sqrt{\lambda_t^T} \EE_t\left[\int_t^T \sqrt{\lambda^T_s} \left| D_{(t,0)}\left(Z^T_s \right) \right| ds \right] dt }.
    \end{align*}
    Moreover, Jensen's inequality yields
    \begin{align*}
        d_W\left( F^T, G \right) & \leq  \frac{C}{\sqrt{T}} + \frac{1}{\sqrt{T}}\E{\int_{0}^T  \sqrt{\lambda_t^T} \sqrt{\EE_t\left[ \left|\delta \left(\z^T D_{(t,0)}\left(  Z^T\right)\right)\right|^2 \right]} dt }\\
        & \quad +\frac{2}{T}\E{\int_{0}^T  \sqrt{\lambda_t^T} \int_t^T \sqrt{\EE_t\left[\lambda^T_s \left|D_{(t,0)}\left(Z^T_s \right) \right|^2 \right]ds } dt }.\\
    \end{align*}
    Using Lemma \ref{lem: prod_div} and Jensen's inequality, we get
    \begin{align*}
        d_W\left( F^T, G \right) & \leq  \frac{C}{\sqrt{T}} + \frac{1}{\sqrt{T}}\E{\int_{0}^T  \sqrt{\lambda_t^T} \sqrt{ \int_t^T \EE_t\left[\lambda^T_s \left|D_{(t,0)}\left(Z^T_s \right) \right|^2 \right]ds }  dt }\\
        & \quad +\frac{2}{T}\E{\int_{0}^T  \sqrt{\lambda_t^T} \int_t^T \sqrt{\EE_t\left[\lambda^T_s \left|D_{(t,0)}\left(Z^T_s \right) \right|^2 \right]ds } dt }\\
        & \leq  \frac{C}{\sqrt{T}} + \frac{1}{\sqrt{T}}\E{\int_{0}^T  \sqrt{\lambda_t^T} \sqrt{ \int_t^T \EE_t\left[\lambda^T_s \left|D_{(t,0)}\left(Z^T_s \right) \right|^2 \right]ds }  dt }\\
        & \quad +\frac{2}{\sqrt{T}}\E{\int_{0}^T  \sqrt{\lambda_t^T} \sqrt{ \int_t^T \EE_t\left[\lambda^T_s \left|D_{(t,0)}\left(Z^T_s \right) \right|^2 \right]ds }  dt }.
    \end{align*}
    Rewriting the terms, we finally have
    \begin{align*}
        d_W\left( F^T, G \right) & \leq  \frac{C}{\sqrt{T}} + \frac{C}{T}\E{\int_{0}^T  \sqrt{\lambda_t^T} \sqrt{ \int_t^T \EE_t\left[\left| 1-\frac{\lambda^T_s}{\lambda^T_s \circ \eps_{(t,0)}^+} \right|^2 \right]ds }  dt }.
    \end{align*}
\end{proof}


\section{Technical Lemmata}
\label{sec:technical}

In this section, we detailed some useful lemmas for the previous proofs. We begin with the link between the Wasserstein distance and the third moment of our quantity $F^T$.

\begin{lemma}
\label{lem: ineq 2}
    Let $T>0$, $v\in \mathrsfs{F}_W$ and $f,g \in \Ccal$. Then, under Assumption \ref{assump: dans S},
    \begin{align*}
        \E { F^T v(F^T)-v'\left( F^T\right)} 
        & \leq \E{\left| 1-\frac{1}{T}\int_{0}^T  \frac{f^T(t)}{g^T(t)} dt \right|} + \E{\left( F^T\right)^3}-\frac{2}{T}\E{ F^T \int_{0}^T \frac{f^T(t)}{g^T(t)} dt}   \\
        &\quad + \E{\left(v'\left(F^T \right)-2F^T \right)\iint_{\RR^2_+}  \z^T_{(t,\theta)} Z^T \delta \left( D_{(t,\theta)}\left( \z^T Z^T\right)\right) dt d\theta}.
    \end{align*}
\end{lemma}
\begin{proof}
    Using integration by part, one can get
    $$\E { F^T v(F^T)} = \E{\delta\left(\z^{T} Z^T\right)v(F^T)} = \E { \iint_{\RR_+^2} \z_{(t,\theta)}^{T} Z_t^T D_{(t,\theta)}\left(v(F^T) \right) dtd\theta}.$$
    Besides, Taylor expansion gives
    $$ D_{(t,\theta)}(v(F^T)) = v'(F^T) D_{(t,\theta)}(F^T) +  \vert D_{(t,\theta)}(F^T) \vert^2 \int_{0}^{1} (1-x)v''\left(F^T + x D_{(t,\theta)}\left(F^T \right) \right) dx .$$
    We have:

    \begin{align*}
        \E { F^T v(F^T)} &= \E{\iint_{\RR^2_+}  \z^T_{(t,\theta)} Z_t^T v'\left(F^T \right) D_{(t,\theta)}\left( F^T\right) dt d\theta}\\
        & \quad + \E{ \iint_{\RR_+^2} \z^T_{(t,\theta)} Z^T \left|D_{(t,\theta)}\left(F^T \right)\right|^2 \left(\int_0^1 v''\left(F^T + u D_{(t,\theta)}\left(F^T \right) \right)(1-u) du \right)dt d\theta}.
    \end{align*}
Moreover, Proposition \ref{prop: heisenberg} for $D_{(t,\theta)}\left( F^T\right)$ yields:
\begin{align*}
        \E { F^T v(F^T)} &= \E{\iint_{\RR^2_+}  \z^T_{(t,\theta)} \left|Z_t^T \right|^2 v'\left(F^T \right) dt d\theta} + \E{\iint_{\RR^2_+} v'\left(F^T \right) \z^T_{(t,\theta)} Z^T \delta \left( D_{(t,\theta)}\left( \z^T Z^T\right)\right) dt d\theta}\\
        & \quad + \E{ \iint_{\RR_+^2} \z^T_{(t,\theta)} Z^T \left|D_{(t,\theta)}\left(F^T \right)\right|^2 \left(\int_0^1 v''\left(F^T + u D_{(t,\theta)}\left(F^T \right) \right)(1-u) du \right)dt d\theta}.
    \end{align*}
Therefore, since $\Vert v'' \Vert_{\infty} \leq 2$,
\begin{align*}
        &\E { F^T v(F^T)-v'\left( F^T\right)} \\
        &= \E{v'\left(F^T \right)\left( 1-\frac{1}{T}\int_{0}^T  \frac{f^T(t)}{g^T(t)} dt \right)} + \E{\iint_{\RR^2_+} v'\left(F^T \right) \z^T_{(t,\theta)} Z^T \delta \left( D_{(t,\theta)}\left( \z^T Z^T\right)\right) dt d\theta}\\
        & \quad + \E{ \iint_{\RR_+^2} \z^T_{(t,\theta)} Z^T \left|D_{(t,\theta)}\left(F^T \right)\right|^2 \left(\int_0^1 v''\left(F^T + u D_{(t,\theta)}\left(F^T \right) \right)(1-u) du \right)dt d\theta} \\
        & \leq \E{v'\left(F^T \right)\left( 1-\frac{1}{T}\int_{0}^T  \frac{f^T(t)}{g^T(t)} dt \right)} + \E{\iint_{\RR^2_+} v'\left(F^T \right) \z^T_{(t,\theta)} Z^T \delta \left( D_{(t,\theta)}\left( \z^T Z^T\right)\right) dt d\theta}\\
        & \quad + \E{ \iint_{\RR_+^2} \z^T_{(t,\theta)} Z^T \left|D_{(t,\theta)}\left(F^T \right)\right|^2 dt d\theta}. \\
    \end{align*}
Now, notice that for $(t,\theta)\in \RR_+^2$,
$$\left|D_{(t,\theta)}\left(F^T \right)\right|^2 = D_{(t,\theta)}\left(\left|F^T \right|^2\right) - 2 F^TD_{(t,\theta)}\left(F^T \right). $$
This last equality gives
\begin{align*}
        &\E { F^T v(F^T)-v'\left( F^T\right)}  \\
        & \leq \E{v'\left(F^T \right)\left( 1-\frac{1}{T}\int_{0}^T  \frac{f^T(t)}{g^T(t)} dt \right)} + \E{\iint_{\RR^2_+} v'\left(F^T \right) \z^T_{(t,\theta)} Z^T \delta \left( D_{(t,\theta)}\left( \z^T Z^T\right)\right) dt d\theta}\\
        & \quad + \E{ \iint_{\RR_+^2} \z^T_{(t,\theta)} Z^T D_{(t,\theta)}\left(\left|F^T \right|^2\right) dt d\theta} -2\E{ \iint_{\RR_+^2} \z^T_{(t,\theta)} Z^T F^T D_{(t,\theta)}\left(F^T \right) dt d\theta}.
    \end{align*}
Besides, we have that
$$\E{\left( F^T \right)^3} = \E{F^T \left|F^T \right|^2} = \E{\delta\left( \z^T Z^T\right)\left|F^T \right|^2}= \E{ \iint_{\RR_+^2} \z^T_{(t,\theta)} Z^T D_{(t,\theta)}\left(\left|F^T \right|^2\right) dt d\theta}.$$
And so
\begin{align*}
\E { F^T v(F^T)-v'\left( F^T\right)} 
        & \leq \E{v'\left(F^T \right)\left( 1-\frac{1}{T}\int_{0}^T  \frac{f^T(t)}{g^T(t)} dt \right)} + \E{\left( F^T\right)^3} \\
        & \quad +  \E{\iint_{\RR^2_+} v'\left(F^T \right) \z^T_{(t,\theta)} Z^T \delta \left( D_{(t,\theta)}\left( \z^T Z^T\right)\right) dt d\theta} \\
        &\quad -2\E{ \iint_{\RR_+^2} \z^T_{(t,\theta)} Z^T F^T D_{(t,\theta)}\left(F^T \right) dt d\theta}
\end{align*}
Now, we make use of the Proposition \ref{prop: heisenberg} to get
\begin{align*}
    & \E{F^T v \left( F^T\right)-v'\left(F^T \right)} \\
    & \leq \E{v'\left(F^T \right)\left( 1-\frac{1}{T}\int_{0}^T  \frac{f^T(t)}{g^T(t)} dt \right)} + \E{\left( F^T\right)^3}\\
    &\quad + \E{\iint_{\RR^2_+} v'\left(F^T \right) \z^T_{(t,\theta)} Z^T \delta \left( D_{(t,\theta)}\left( \z^T Z^T\right)\right) dt d\theta}\\
    & \quad - 2\E{ \iint_{\RR_+^2} \z^T_{(t,\theta)} \left|Z^T\right|^2 F^T dt d\theta}- 2\E{ \iint_{\RR_+^2} \z^T_{(t,\theta)} Z^T F^T \delta\left(D_{(t,\theta)}\left(\z^T Z^T \right) \right)dt d\theta}.
\end{align*}
Thus, by ordering the terms, we finally have
\begin{align*}
    & \E{F^T v \left( F^T\right)-v'\left(F^T \right)}\\
    &\leq \E{v'\left(F^T \right)\left( 1-\frac{1}{T}\int_{0}^T  \frac{f^T(t)}{g^T(t)} dt \right)} + \E{\left( F^T\right)^3}-2\E{ \iint_{\RR_+^2} \z^T_{(t,\theta)} \left|Z^T\right|^2 F^T  dt d\theta} \\
    &\quad + \E{\left(v'\left(F^T \right)-2F^T \right)\iint_{\RR^2_+}  \z^T_{(t,\theta)} Z^T \delta \left( D_{(t,\theta)}\left( \z^T Z^T\right)\right) dt d\theta}\\
    & \leq \E{\left| 1-\frac{1}{T}\int_{0}^T  \frac{f^T(t)}{g^T(t)} dt \right|} + \E{\left( F^T\right)^3}-\frac{2}{T}\E{ F^T \iint_{0}^T \frac{f^T(t)}{g^T(t)} dt}   \\
    &\quad + \E{\left(v'\left(F^T \right)-2F^T \right)\iint_{\RR^2_+}  \z^T_{(t,\theta)} Z^T \delta \left( D_{(t,\theta)}\left( \z^T Z^T\right)\right) dt d\theta}.
\end{align*}
\end{proof}

The two next lemmas are about the divergence operator: Lemma \ref{lem: magique} simplifies the computations and Lemma \ref{lem: prod_div} transforms a product of two divergences into an integral. In particular, Lemma  \ref{lem: magique} refines Lemma 4.1 of \cite{khabou_normal_2024} which allow us to simplify some terms.

\begin{lemma}
\label{lem: magique}
    Let $F$ be in $\mathbb L^2 (\Omega , \mathbb F_{\infty}^N, \mathbb P)$. We set $\z_{(t,\theta)} := \1_{\theta \leq f(t)}$ where $(t,\theta) \in \RR^2_+$ and $f$ a predictable process.
    Suppose that for any $(s,t,\theta, \rho) \in \RR^4_+$, 
    $$D_{(t,\theta)} f(s) \geq 0 \, \text{,} \quad \1_{\rho \geq f(s)}D_{(s,\rho)} F = 0 \quad \text{and} \quad u D_{(t,\theta)}\left( \z\right)  \in \Scal.$$
    Then, 
    $$\E{ \z_{(t,\theta)} F \delta \left( u D_{(t,\theta)}\left( \z\right)  \right) } =0.$$
\end{lemma}

\begin{proof}
    Using the integration by part formula,
    $$\E{ \z_{(t,\theta)} F \delta \left( D_{(t,\theta)}\left( \z\right) u \right) } = \E{ \int_t^{+\infty}\int_{\RR_+} u_s D_{(t,\theta)}\left( \z_{(s,\rho)}\right) D_{(s,\rho)}\left( \z_{(t,\theta)} F \right)d\rho ds}.$$
    Moreover, the properties of Malliavin's derivative ensures that
    \begin{align*}
      D_{(t,\theta)}\left( \z_{(t,\theta)} F\right) &= \z_{(t,\theta)} D_{(t,\theta)}\left( F\right)+F D_{(t,\theta)}\left( \z_{(t,\theta)} \right) + D_{(t,\theta)}\left( F\right)D_{(t,\theta)}\left( \z_{(t,\theta)} \right).
    \end{align*}
    Hence,
    \begin{align*}
        \E{ \z_{(t,\theta)} F \delta \left( D_{(t,\theta)}\left( \z\right) u \right) } &= \E{ \int_t^{+\infty}\int_{\RR_+} u_s D_{(t,\theta)}\left( \z_{(s,\rho)}\right) D_{(s,\rho)}\left( \z_{(t,\theta)} \right) F\circ\eps^+_{(s,\rho)}d\rho ds}\\
        &+ \E{ \int_t^{+\infty}\int_{\RR_+} u_s D_{(t,\theta)}\left( \z_{(s,\rho)}\right) D_{(s,\rho)}\left(F \right)\z_{(t,\theta)} d\rho ds}\\
        &= \E{ \int_t^{+\infty}\int_{\RR_+} u_s D_{(t,\theta)}\left( \z_{(s,\rho)}\right) D_{(s,\rho)}\left(F \right)\z_{(t,\theta)} d\rho ds}.
    \end{align*}
    Besides, by noticing that $D_{(t,\theta)}\left( \z_{(s,\rho)}\right) D_{(s,\rho)}\left( \z_{(t,\theta)} \right)=0$, we get
        $$\E{ \z_{(t,\theta)} F \delta \left( D_{(t,\theta)}\left( \z\right) u \right) } = \E{ \int_t^{+\infty}\int_{\RR_+} u_s D_{(t,\theta)}\left( \z_{(s,\rho)}\right) D_{(s,\rho)}\left(F \right)\z_{(t,\theta)} d\rho ds}.$$
    Now, since $D_{(t,\theta)} f(s) \geq 0$, we have that
    $$D_{(t,\theta)}\left( \z_{(s,\rho)}\right) = \1_{f(s) < \rho \leq f(s)\circ \eps^+_{(t,\theta)}} .$$
    Therefore, the second assumption ensures that
        $$\E{ \z_{(t,\theta)} F \delta \left( D_{(t,\theta)}\left( \z\right) u \right) } = \E{ \int_t^{+\infty}\int_{\RR_+} u_s \1_{f(s) < \rho \leq f(s)\circ \eps^+_{(t,\theta)}} \z_{(t,\theta)} D_{(s,\rho)}\left( F \right)d\rho ds} =0.$$
\end{proof}

\begin{lemma}
\label{lem: prod_div}
    Let $u$ and $\hat{u}$ be in $\Scal$. We set $\1_{t}:s \mapsto \1_{t\leq s}$ where $t \in \RR$. Then,
    $$\E{\delta\left( u \right)\delta\left( \hat{u} \1_{t} \right)} = \int_{t}^{+\infty}\int_{\RR_+} \E{u_{(s,\theta)} \hat{u}_{(s,\theta)}} ds d\theta .$$
\end{lemma}

\begin{proof}
    Let $t\in \RR$. Then, Proposition \ref{prop:IPP} gives
    $$\E{\delta\left( u \right)\delta\left( \hat{u} \1_{t} \right)} = \int_t^{+\infty} \int_{\RR_+} \E{D_{(s,\rho)}\left(\delta (u) \right) \hat{u}_{(s,\rho)}} dsd\rho.$$
    Moreover, using Proposition \ref{prop: heisenberg}, we have
    \begin{align*}
      \E{\delta\left( u \right)\delta\left( \hat{u} \1_{t} \right)} &= \int_t^{+\infty} \int_{\RR_+} \E{u_{(s,\rho)} \hat{u}_{(s,\rho)}} dsd\rho + \int_t^{+\infty} \int_{\RR_+} \E{\delta\left( D_{(s,\rho)}(u) \right) \hat{u}_{(s,\rho)}} dsd\rho \\
      &= \int_t^{+\infty} \int_{\RR_+} \E{u_{(s,\rho)} \hat{u}_{(s,\rho)}} dsd\rho.
    \end{align*}
    Here, we have use the fact that $\EE_s\left[\delta\left( D_{(s,\rho)}(u) \right) \right]=0$.
\end{proof}

\begin{lemma}
\label{lem: maj_Dlambda12}
    Let $f = \left( f^T(t)\right)_{t\in [0,T]} \in \Ccal$ such that $D_{(t,0)}\left( f^T(s)\right) \geq 0$. Then, for $0\leq t \leq s \leq T$, 
    $$\left\vert D_{(t,0)}\left( \frac{1}{\sqrt{f^T(s)}} \right) \right\vert \leq \frac{1}{2} \frac{D_{(t,0)}\left( f^T(s) \right)}{ \left( f^T(s)\right)^{3/2}}.$$
\end{lemma}

\begin{proof}
Note that $D_{(t,0)} \left( f^T(s)^{-1/2}\right) \leq 0$.\\
Letting $\bar f^T(s)$ a random element in $[f^T(s),f^T(s) \circ \eps_{(t,0)}^+]$ we have
\begin{align*}
|D_{(t,0)} (f^T(s)^{-1/2})|&=-D_{(t,0)} \left(f^T(s)^{-1/2}\right) \\
&= -\left(\frac{1}{\sqrt{f^T(s) \circ \eps_{t}^+}} - \frac{1}{\sqrt{f^T(s)}}\right) \\
&= \frac12 (f^T(s))^{-3/2} D_{(t,0)} f^T(s) - \frac38 \bar f^T(s)^{-5/2} |D_{(t,0)} f^T(s)|^2 \\
& \leq \frac12 f^T(s)^{-3/2} D_{(t,0)} f^T(s).
\end{align*}
\end{proof}

\begin{lemma}
\label{lemma: maj_esp_cond_delta}
    Take $f=g \in \Ccal$ such that Assumption \ref{assump: dans S} and \ref{assump: initiale} hold. Then, for any $s$ in $[0,T]$,
    $$\EE_s \left[ \vert \delta(\z^T D_{(s,0)}(Z^T) \vert  \right] \leq  \frac{2}{\sqrt{T}} \int_{s}^T \EE_s \left[  \frac{\left| D_{(s,0)}\left(f^T(u) \right)\right|}{\sqrt{f^T(u)}}\right] du.$$
\end{lemma}

\begin{proof} 
Let $s$ be in $[0,T]$. By the triangular inequality, we have
\begin{align*}
    \EE_s\left[|\delta\left(\z^T D_{(s,0)} Z^T)\right)|\right] &\leq \frac{1}{\sqrt{T}} \EE_s \left[ \iint_{(s,T]\times \mathbb{R}_+} \mathds{1}_{\{\theta \leq f^T(u)\}} \left|D_{(s,0)} \left( \frac{1}{\sqrt{f^T(u)}} \right|\right) N(du,d\theta)\right] \\
    &\quad + \frac{1}{\sqrt{T}} \EE_s \left[ \iint_{(s,T]\times \mathbb{R}_+} \mathds{1}_{\{\theta \leq f^T(u)\}} \left|D_{(s,0)} \left( \frac{1}{\sqrt{f^T(u)}} \right|\right) dud\theta\right]\\
    & \leq \frac{2}{\sqrt{T}} \EE_s \left[ \iint_{(s,T]\times \mathbb{R}_+} \mathds{1}_{\{\theta \leq f^T(u)\}} \left|D_{(s,0)} \left( \frac{1}{\sqrt{f^T(u)}} \right|\right) dud\theta\right].
\end{align*}
Moreover, integrating on $\theta$ and using Lemma \ref{lem: maj_Dlambda12}, we get
\begin{align*}
    \EE_s\left[|\delta\left(\z^T D_{(s,0)} Z^T)\right)|\right] &\leq \frac{1}{\sqrt{T}} \EE_s \left[ \iint_{(s,T]\times \mathbb{R}_+} \mathds{1}_{\{\theta \leq f^T(u)\}} \left|D_{(s,0)} \left( \frac{1}{\sqrt{f^T(u)}} \right|\right) N(du,d\theta)\right] \\
    &\quad + \frac{1}{\sqrt{T}} \EE_s \left[ \iint_{(s,T]\times \mathbb{R}_+} \mathds{1}_{\{\theta \leq f^T(u)\}} \left|D_{(s,0)} \left( \frac{1}{\sqrt{f^T(u)}} \right|\right) dud\theta\right]\\
    & \leq \frac{2}{\sqrt{T}} \EE_s \left[ \int_{(s,T]} \frac{\left| D_{(s,0)}\left(f^T(u) \right)\right|}{\sqrt{f^T(u)}}  du\right].
\end{align*}
\end{proof}

In the following lemma, we compute in a special case the third moment of $F^T = \delta\left( \z^T Z^T\right)$ with $Z^T_t = \frac{\1_{t\in [0,T]}}{\sqrt{T f^T(t)}}$ and $\z^T_{(t,\theta)} = \ind{\theta \leq f^T(t)}$ for $(t,\theta) \in \RR_+^2$.
\begin{lemma}
\label{lem: moment 3}
    For $f=g \in \Ccal$ with $\z_{(t,\theta)}^T D_{(t,\theta)}\left( Z^T\right)= \z_{(t,\theta)}^T D_{(t,0)}\left( Z^T\right)$ for any $T>0$, we have:
    $$\E{\left( F^T \right)^3} = \frac{1}{T^{3/2}} \int_0^T \E{\frac{1}{\sqrt{f^T(t)}}} dt.$$
\end{lemma}

\begin{proof}
    Let $T>0$ be fixed. Using the integration by part formula, we have
\begin{align*}
\E{\left(F^T \right)^3} &= \E{F^T \left(F^T \right)^2} = \iint_{\RR_+^2}\E{ \z^T_{(t,\theta)} Z^T_t D_{(t,0)} \left( \left|F^T \right|^2 \right)} dt d\theta \\
&= \frac{1}{\sqrt{T}}\int_0^T \E{ \sqrt{f^T(t)} D_{(t,0)} \left( \left|F^T \right|^2 \right)} dt. 
\end{align*}
Moreover, for $t\in [0,T]$, we have
\begin{align*}
    D_{(t,0)} \left( \left|F^T \right|^2 \right) &= 2 F^T D_{(t,0)} \left(F^T \right) + \left|D_{(t,0)} \left(F^T \right) \right|^2.
\end{align*}
 Thus, Heisenberg's equality (see Proposition \ref{prop: heisenberg}) yields
\begin{align*}
    D_{(t,0)} \left( \left|F^T \right|^2 \right) &= 2 \left\{ Z_t^T + \delta \left(D_{(t,0)} \left(\z^T Z^T\right) \right)\right\}\delta \left( \z^T Z^T\right) + \left\{ Z_t^T + \delta \left(D_{(t,0)} \left(\z^T Z^T\right) \right)\right\}^2\\
    &= \left[\left| Z_t^T \right|^2 + 2 Z_t^T \delta \left(\z^T Z^T \right) \right]+ 2 Z_t^T \delta \left(D_{(t,0)} \left(\z^T Z^T\right)\right) \\
    & \quad + \delta \left(D_{(t,0)} \left(\z^T Z^T\right)\right) \delta\left( 2\z^T Z^T+D_{(t,0)} \left(\z^T Z^T\right) \right).
\end{align*}
Let us denote by $A_1$, $A_2$ and $A_3$ the right hand terms of the equality, i.e.
\begin{align*}
    A_1 &:= \left| Z_t^T \right|^2 + 2 Z_t^T \delta \left(\z^T Z^T \right); \\
    A_2 &:= 2 Z_t^T \delta \left(D_{(t,0)} \left(\z^T Z^T\right)\right); \\
    A_3 &:= \delta \left(D_{(t,0)} \left(\z^T Z^T\right)\right) \delta\left( 2\z^T Z^T+D_{(t,0)} \left(\z^T Z^T\right) \right).
\end{align*}
Now, note that $\EE_t \left[A_1 \right] = \left| Z^T_t \right|^2$ and $\EE_t \left[A_2 \right] =0$. Let us prove that we also have $\EE_t \left[A_3 \right] =0$.\\
Take $(s,\rho) \in \RR_+^2$. Then, by rearranging the terms, we get
\begin{align*}
    2\z_{(s,\rho)}^T Z_s^T+D_{(t,0)} \left(\z_{(s,\rho)}^T Z_s^T\right) &= \frac{1}{\sqrt{T}} \left(\frac{\ind{f^T(s) \leq \rho \leq f^T(s) \circ \eps^{+}_{(t,0)}}}{\sqrt{f^T(s) \circ \eps^{+}_{(t,0)}}}+ \ind{\rho \leq f^T(s)} \left[\frac{1}{\sqrt{\lambda^T_s\circ \eps^{+}_{(t,0)}}} + \frac{1}{\sqrt{f^T(s)}}\right]\right).
\end{align*}
Therefore, using Lemma \ref{lem: prod_div}, we can write
\begin{align*}
    \EE_t\left[ A_3\right] &=\frac{1}{\sqrt{T}} \EE_t \left[\int_t^T \int_{\RR^*_+} D_{(t,0)} \left(\z_{(s,\rho)}^T Z_s^T\right) \left\{\frac{\ind{f^T(s) \leq \rho \leq f^T(s) \circ \eps^{+}_{(t,0)}}}{\sqrt{f^T(s) \circ \eps^{+}_{(t,0)}}}\right. \right.\\
    &\quad + \left. \left.\ind{\rho \leq f^T(s)} \left[\frac{1}{\sqrt{f^T(s)\circ \eps^{+}_{(t,0)}}} - \frac{1}{\sqrt{f^T(s)}}\right] \right\} ds d\rho \right].
\end{align*}
We now develop the terms in the integral which provides
$$ \EE_t\left[ A_3\right]= \frac{1}{T} \EE_t \left[\int_t^T \int_{\RR^*_+} \frac{\ind{f^T(s) \leq \rho \leq f^T(s) \circ \eps^{+}_{(t,0)}}}{f^T(s) \circ \eps^{+}_{(t,0)}}  + \ind{\rho \leq f^T(s)} \left[\frac{1}{f^T(s)\circ \eps^{+}_{(t,0)}} - \frac{1}{f^T(s)}\right]ds d\rho \right]. $$
Finally, by integrating according to $\rho$, we have
$$\EE_t\left[ A_3\right]= \frac{1}{T} \EE_t \left[ \int_t^T \frac{f^T(s) \circ \eps^+_{(t,0)} - f^T(s)}{f^T(s) \circ \eps^{+}_{(t,0)}} + f^T(s) \left[\frac{1}{f^T(s)\circ \eps^{+}_{(t,0)}} - \frac{1}{f^T(s)}\right] ds\right]= 0 .$$

Therefore, $\EE_t\left[ D_{(t,0)} \left( \left|F^T \right|^2 \right) \right] = \left| Z^T_t\right|^2$ and so
$$\E{ \left( F^T\right)^3} = \frac{1}{\sqrt{T}} \int_0^T \E{\sqrt{f^T(t)}  \left| Z^T_t\right|^2} dt = \frac{1}{T^{3/2}} \int_0^T \E{\frac{1}{\sqrt{f^T(t)}}} dt.$$

\end{proof}

\appendix
\section{(Stochastic) Volterra's Equation}
\label{app : volterra}

We write in this appendix some results about the solution or sub-solution of Volterra's equation. We believe these results are known but we have not found paper with our setting. That is why we decide to provide proofs in this appendix. We split it in two different parts: the first part is about the classical Volterra's equation and the second deal with a stochastic version of such equation.

\subsection{Some elements on deterministic Volterra's equation}
\begin{lemma}[Lemma 3, \cite{bacry_limit_2013}]
    Let $M$ be a locally bounded function from $\RR_+$ to $\RR$. For $\Phi: \RR_+ \to \RR$ such that $\Vert \Phi \Vert_1 <1$, there exists a unique locally bounded function $L: \RR_+ \to \RR$ solution to
    \begin{equation}
    \label{eq: volterra eq}
        L(t) = M(t) + \int_0^t \Phi(t-s) L(s) ds, \; t \geq 0
    \end{equation}
    given by
    $$ L(t) = M(t) + \int_0^t \Psi(t-s) M(s) ds$$
    where $\Psi = \sum\limits_{k\geq 1} \Phi^{\ast k}$ and $\Phi^{\ast (k+1)} = \Phi \ast \Phi^{\ast k}$, $k\geq 1$.
\end{lemma}

The positivity of the solution and a comparison of solutions of \eqref{eq: volterra eq} is given in the two following lemmas.

\begin{lemma}
\label{lemma:positiveVolterra}
Let $\Phi \geq 0$ with $\|\Phi\|_1 < 1$. Let $M:\RR_+ \to \RR_+$ be locally bounded. Then the unique locally bounded solution $L$ to \eqref{eq: volterra eq}
is such that $L(t) \geq 0$, for all $t\geq 0$. 
\end{lemma}

\begin{proof}
As $\Phi\geq 0$ we have the stronger result that $L(t) \geq M(t)$ if $M\geq 0$ as
$$ L(t) = M(t) + \int_0^t \Psi(t-s) M(s) ds.$$
\end{proof}

\begin{lemma}
\label{lemma:positiveVolterracomparison}
Let $\Phi \geq 0$ with $\|\Phi\|_1 < 1$. Let $M_1, M_2:\mathbb R_+ \to \mathbb R_+$ both locally bounded. Let $L_i$ the unique locally bounded solution to : 
$$ L_i(t) = M_i(t) + \int_0^t \Phi(t-s) L_i(s) ds, \; t \in [0,T], \; i=1,2.$$
Then,
$$ \left[M_1(t) \geq M_2(t), \; \forall t\geq 0\right] \Rightarrow \left[L_1(t) \geq L_2(t), \; \forall t\geq 0\right].$$
\end{lemma}

\begin{proof}
Apply Lemma \ref{lemma:positiveVolterra} to $M(t):=L_1(t)-L_2(t) \geq 0$.
\end{proof}

\begin{proposition}
\label{prop: sous_sol volterra}
Let $\Phi \geq 0$ with $\|\Phi\|_1 < 1$. Let $M:\mathbb R_+ \to \mathbb R_+$ be locally bounded and $l:\mathbb R_+ \to \mathbb R_+$ in $L_1(\mathbb R_+)$ such that:
$$ l(t) \leq M(t) + \int_0^t \Phi(t-s) l(s) ds, \quad t\geq 0.$$
Then
$$ l(t) \leq M(t)+\int_0^t \Psi(t-s) M(s) ds.$$
\end{proposition}

\begin{proof}
Let $m(t):=l(t)-\int_0^t \Phi(t-s) l(s) ds$. So $m(t) \leq M(t)$ for all $t$. Let $L$ be the unique solution to 
$$ L(t) = M(t) + \int_0^t \Phi(t-s) L(s) ds, \; t \geq 0.$$
We have that $L(t) = M(t) + \int_0^t \Psi(t-s) M(s) ds$ and
$$l(t)-L(t) = (m(t)-M(t))+\int_0^t \Phi(t-s) (l(s)-L(s)) ds $$
and since $m(t)-M(t)\leq 0$ for any $t$, Lemma \ref{lemma:positiveVolterra} gives $l(t)-L(t) \leq 0$ for $t\geq 0$.
\end{proof}

\subsection{Stochastic Volterra's equation}

We now present a generalisation of Theorem 2.7 in \cite{hillairet_expansion_2023} by introducing non-linearity.
The non-linear property may alter some elements in the existence part, which we elaborate on in the proof. Many proofs exists in the literature. Yet, our proof is based on the Poisson Imbedding and relies on the method used by Ogata \cite{ogata_lewis_1981}. Moreover, we have chosen not to present the uniqueness part here and instead refer to \cite{hillairet_expansion_2023} for an explicit proof.

Let $\GG = \left(\Gcal_s \right)_{s\in \RR_+}$ be a filtration such that $\FF^N \subset \GG$ and such that $N$ is a $\GG$ Poisson measure.

\begin{proposition}
\label{prop: existence}
    Let $t_0 \in \RR_+$, $\phi:\RR_+ \to \RR$ be locally bounded, $h:\RR \to \RR_+$ be such that $h$ is $\alpha$-Lipschitz.
    In addition, let 
    $(\mu_t)_{t\geq t_0}$ be a $\GG$-predictable process such that $t\mapsto \E{h(\mu_t) \mid \Gcal_0}$ is locally bounded.
    Then, the SDE below admits a unique $\GG$-predictable solution $\Lambda$. 
    \begin{equation}
    \label{eq: volterra general}
    \Lambda_t = h \left(\mu_t + \iint_{(t_0,t) \times \RR_+} \phi(t-u)\ind{\theta \leq \Lambda_u} dN(u,\theta) \right) ,\quad t \geq t_0.
    \end{equation}
    Moreover, for any $t\geq t_0$, we have
    $$\E{\Lambda_t \mid \Gcal_0} \leq \E{h(\mu_t)\mid \Gcal_0} + \int_{t_0}^t \psi^{(\alpha)}(t-s) \E{h(\mu_s)\mid \Gcal_0} ds $$
    where $\psi^{(\alpha)} = \sum\limits_{k\geq 1} \alpha^k \left|\phi\right|^{\ast k} $ and $\left|\phi\right|^{\ast (k+1)} = \left|\phi\right| \ast \left|\phi\right|^{\ast k}$, for $k\geq 1$.
\end{proposition}

\begin{proof}
    For $t\geq t_0$, we set $\Lambda^{(1)}_t$ and $H^{(1)}_t$ as follow
$$\Lambda^{(1)}_t := h(\mu_t) \quad \text{and} \quad H_t^{(1)} := \iint_{(t_0, t]\times \RR_+} \ind{\theta \leq \Lambda_s^{(1)}} N(ds, d\theta).$$
Moreover, we set $\tau^H_1$ the first time jump of $H^{(1)}$, i.e.
$$\tau^H_1 := \inf\left\{ \tau >t_0 \mid H_\tau^{(1)} =1\right\}$$
Then, $\tau_1^{H}$ is a $\GG$ stopping time. The sequence $\left( \Lambda^{(n)}, H^{(n)}, \tau^H_n\right)_{n\in \NN^*}$ is constructed by induction with:
\begin{align*}
        \Lambda^{(n+1)}_t &:= h\left( \mu_t + \iint_{(t_0,t)\times \RR_+} \1_{s\in [0,\tau^H_n]} \phi(t-s) \ind{\theta \leq \Lambda_s^{(n)}} N(ds,d\theta)\right), \quad t\geq 0;\\
        H^{(n+1)}_t &:= \iint_{(t_0, t]\times \RR_+} \ind{\theta \leq \Lambda_s^{(n+1)}} N(ds, d\theta), \quad t\geq 0;\\
        \tau_{n+1}^{H} &:= \inf \left\{ \tau >t_0 \mid H^{(n+1)}_\tau = n+1 \right\}.
\end{align*}
Using an induction reasoning, one can prove that $(\tau_n^{H})_{n\in \NN^*}$ is a sequence of $\GG$ stopping times and $\left( \Lambda^{(n)} \right)_{n\in \NN^*}$ is a sequence of $\Gcal_{\tau_n^H}$-measurable and predictable stochastic processes. Note that we have $\Lambda^{(n)}=\Lambda^{(n+1)}$ on $[t_0, \tau_n^{H}]$ for any $n \in \NN^*$.
Moreover, we can prove by induction that for any $n\in \NN^*$, 
$$\E{\Lambda^{(n)}_t \mid \Gcal_0} \leq \E{h(\mu_t) \mid \Gcal_0} + \int_{t_0}^t \psi^{(\alpha)}(t-s) \E{h(\mu_s)\mid \Gcal_0} ds .$$
It suffices to remark that $t\mapsto \E{h(\mu_t) \mid \Gcal_0}$ is locally bounded and that for $t\geq t_0$ and $n\in \NN^*$ we have
\begin{align*}
    \E{\Lambda^{(n+1)}_t \mid \Gcal_0 } &\leq \E{h(\mu_t)\mid \Gcal_0} + \alpha \int_{t_0}^t \left|\phi(t-s)\right| \E{\Lambda^{(n)}_s \mid \Gcal_0} ds.
\end{align*}

We now set $H$ and $\Lambda$ as the limits of $H^{(n)}$ and $\Lambda^{(n)}$, i.e. for any $t\geq t_0$,
$$H_t := \sum\limits_{k\geq 1} \ind{t\geq \tau^H_k} \quad \text{and} \quad \Lambda_t := \lim\limits_{n\to +\infty} \Lambda^{(n)}_t .$$

Finally we prove that $\lim\limits_{n\to +\infty} \tau^H_n = +\infty$. Take $\mathfrak{t}>0$, then by construction 
    \begin{align*}
        {\mathbb P}\left(\tau^\ast <\mathfrak{t} \right) = \limsup_{n\to +\infty}\PP \left( \tau_n^H <\mathfrak{t}\right) \leq \limsup_{n\to +\infty}\PP \left( H_{\mathfrak{t}}\geq n\right).
        \end{align*}
    Using Markov's inequality, we get
    $$\PP\left( \tau^\ast \leq \mathfrak{t}\right) \leq \limsup_{n\to +\infty}\frac{\E{H_{\mathfrak{t}}}}{n} = \limsup_{n\to +\infty} \frac{\int_0^{\mathfrak{t}}\E{\Lambda_s} ds}{n}.$$
    Finally, $s\mapsto \E{\Lambda_s}$ is locally bounded since $s\mapsto \E{h(\mu_s) \mid \Gcal_0}$ is, and so
    $\limsup_{n\to +\infty} \frac{\int_0^{\mathfrak{t}}\E{\Lambda_s} ds}{n} = 0$.
    Thus, $\PP \left(\tau^\ast = +\infty \right)= 1$.
    
\end{proof}

For $\GG = \FF^N$, we have this classical result:
\begin{corollary}
    Let $\phi:\RR_+ \to \RR$ be locally bounded, and $h:\RR \to \RR_+$ be such that $h$ is $\alpha$-Lipschitz and
    $$\alpha \int_{0}^{+\infty} \left| \phi(t) \right| dt <1 .$$ Consider $\lambda$ the solution of 
    $$\lambda_t = h \left(\mu + \iint_{(0,t) \times \RR_+} \phi(t-u)\ind{\theta \leq \lambda_u} dN(u,\theta) \right) ,\quad t \in \RR_+ .$$
    Then,
    $$\sup\limits_{t\geq 0}\E{\lambda_t} \leq \frac{h(\mu)}{1-\alpha \Vert \phi \Vert_1}.$$
\end{corollary}

For the next results, we recall some notations with some assumptions (see Notation \ref{notation: xi} and Assumption \ref{assump: noyau}). We consider $\lambda^\infty$ as the intensity of a stationary non-linear Hawkes process (see Definition \ref{def:Hawkes}), i.e.
$$\lambda^\infty_t = h\left( \mu + \iint_{(-\infty,t)\times \RR_+} \ind{\theta \leq \lambda^\infty_s}\phi(t-s) N(ds,d\theta)\right) = h\left(\mu + \xi^\infty_t \right), \quad t\in \RR .$$
We also consider $\lambda$ as the intensity of a Hawkes process with empty past, i.e.
$$\lambda_t = h\left( \mu + \iint_{(0,t)\times \RR_+} \ind{\theta \leq \lambda_s}\phi(t-s) N(ds,d\theta)\right) = h\left(\mu + \xi_t \right), \quad t\in \RR .$$
Moreover, we suppose that $h: \RR \to \RR_+$ is $\alpha$-Lipschitz. We also suppose that $\alpha$ and $\phi: \RR_+ \to \RR$ are such that $\alpha \Vert \phi \Vert_1 <1$. 

For the next corollary, we rely on \cite{bremaud_stability_1996} for the existence of a stationary Hawkes processes. Indeed, we based the following corollary on the fact that for $t\in \RR$, $\E{\lambda^\infty_t}=\E{\lambda^\infty_0}$.
\begin{corollary}
     Let $\phi:\RR_+ \to \RR$ be locally bounded, and $h:\RR \to \RR_+$ be such that $h$ is $\alpha$-Lipschitz and
    $$\alpha \int_{0}^{+\infty} \left| \phi(t) \right| dt <1 .$$ Consider $\lambda^\infty$ the solution of 
    $$\lambda^\infty_t = h \left(\mu + \iint_{(-\infty,t) \times \RR_+} \phi(t-u)\ind{\theta \leq \lambda^\infty_u} N(du,d\theta) \right) ,\quad t \in \RR .$$
    Then, for all $t\geq s$, we have that $t\mapsto \EE_s \left[ \lambda^\infty_{t}\circ \eps_{(s,0)}^+ \right]$ is locally bounded a.s.
\end{corollary}
\begin{proof}
    Let $t\geq s$ be fixed. Then we have:
    \begin{align*}
    \lambda^\infty_t \circ \eps_{(s,0)}^+ &= h \left(\mu + \phi(t-s) + \iint_{(-\infty,s)\times \RR_+} \ind{\theta \leq \lambda_r^\infty} \phi(t-r) N(dr,d\theta) \right.\\
    &\quad +\left.\iint_{(s,t)\times \RR_+} \ind{\theta \leq \lambda_r^\infty \circ \eps_{(s,0)}^+} \phi(t-r) N(dr,d\theta)  \right).
    \end{align*}
    We set $\mu_t = \mu + \phi(t-s) + \iint_{(-\infty,s)\times \RR_+} \ind{\theta \leq \lambda_r^\infty} \phi(t-r) N(dr,d\theta)$ in order to get
    $$\lambda^\infty_t \circ \eps_{(s,0)}^+ = h \left(\mu_t +\iint_{(s,t)\times \RR_+} \ind{\theta \leq \lambda_r^\infty \circ \eps_{(s,0)}^+} \phi(t-r) N(dr,d\theta)  \right).$$
    Besides, since $\phi$ is locally bounded, we get that
    $$\E{ \mu_t} = \mu + \phi(t-s) + \int_{-\infty}^s \phi(t-r) \E{\lambda^{\infty_r}} dr$$
    is locally bounded.
    Hence, $\left( \mu_t\right)_{t\geq s}$ is $\GG = \left(\Fcal^N_{t+s} \right)_{t\in \RR_+}$-predictable and $t \mapsto \E{h(\mu_t) \mid \Gcal_0}$ is locally bounded. So, we apply Proposition \ref{prop: existence} and we obtain
    $$\EE_s \left[\lambda^\infty_t \circ \eps_{(s,0)}^+ \right] \leq \EE_s \left[h\left( \mu_t\right) \right] + \int_s^t \psi^{(\alpha)}(t-r) \EE_s \left[h\left( \mu_r\right) \right] dr, \quad \forall t\geq s.$$
    Therefore, $t \mapsto \EE_s \left[\lambda^\infty_t \circ \eps_{(s,0)}^+ \right]$ is locally bounded.
\end{proof}

The next lemma links the results about Volterra's equation with our stochastic version.
\begin{lemma} \label{lem: maj esp cond}
    Under Assumption \ref{assump: noyau}, there exists a positive integrable function $\psi^{(\alpha)}$ such that for any $(u,t,\rho) \in \RR_+^3$
    $$\EE_u\left[\left| D_{(u,\rho)}\left(\lambda_t^\infty \right) \right| \right] \leq \1_{\rho \leq \lambda_u^{\infty}}\psi^{(\alpha)}(t-u) \quad \text{and} \quad \EE_u\left[\left| D_{(u,\rho)}\left(\lambda_t \right) \right| \right] \leq\1_{\rho \leq \lambda_u} \psi^{(\alpha)}(t-u), \quad a.s.$$
\end{lemma}

\begin{proof}[Proof of Lemma \ref{lem: maj esp cond}]
    By the definition of the Malliavin derivative, we have
    $$D_{(u,\rho)}\left(\xi_t^\infty \right) = \phi(t-u)\1_{\rho \leq \lambda_u^{\infty}} + \iint_{(u,t)\times \RR_+} \phi(t-r) \left(\1_{\theta \leq \lambda_r^{\infty}\circ \eps_{(u,\rho)}^{+}} - \1_{\theta\leq \lambda_r^{\infty}} \right) N(dr,d\theta ).$$
    Using the triangle inequality, we obtain
    $$\left| D_{(u,\rho)}\left(\xi_t^\infty \right) \right| \leq \left|\phi(t-u)\right|\1_{\rho \leq \lambda_u^{\infty}} + \iint_{(u,t)\times \RR_+} \left|\phi(t-r)\right| \left|\1_{\theta\leq \lambda_r^{\infty}} - \1_{\theta \leq \lambda_r^{\infty}\circ \eps_{(u,\rho)}^{+}} \right| N(dr,d\theta ).$$
    Then,
    $$\EE_u\left[\left| D_{(u,\rho)}\left(\xi_t^\infty \right) \right| \right] \leq \left|\phi(t-u)\right|\1_{\rho \leq \lambda_u^{\infty}} + \int_u^t \left|\phi(t-r)\right| \EE_u\left[\left| D_{(u,\rho)}\left(\lambda_r^\infty \right) \right| \right] dr.$$
    Finally, we have that $x \mapsto \1_{u\leq x}\EE_u\left[\left|D_{(u,\rho)}\left(\lambda_x^\infty \right) \right| \right]$  is a subsolution of a Volterra equation since
    \begin{align*}
    \EE_u\left[\left| D_{(u,\rho)}\left(\lambda_t^\infty \right) \right| \right] &\leq \alpha \EE_u\left[\left| D_{(u,\rho)}\left(\xi_t^\infty \right) \right| \right]\\
    &\leq  \alpha \left|\phi(t-u)\right|\1_{\rho \leq \lambda_u^{\infty}} + \alpha \int_u^t \left|\phi(t-r)\right| \EE_u\left[\left| D_{(u,\rho)}\left(\lambda_r^\infty \right) \right| \right] dr.
    \end{align*}
    According to Proposition \ref{prop: sous_sol volterra}, we define $\psi^{(\alpha)}$ has follows
    $$\psi^{(\alpha)} : x \mapsto \sum\limits_{k\geq 1} \alpha^k \left|\phi \right|^{\ast k}(x)$$
    where $\left| \phi\right|^{\ast (k+1)} = \left| \phi\right| \ast \left| \phi\right|^{\ast k}$ and get
\begin{align*}
    \EE_u\left[\left| D_{(u,\rho)}\left(\lambda_t^\infty \right) \right| \right] &\leq  \alpha \left|\phi(t-u)\right|\1_{\rho \leq \lambda_u^{\infty}} + \alpha \int_u^t \psi^{(\alpha)}(t-r)\left|\phi(r-u)\right|\1_{\rho \leq \lambda_u^{\infty}} dr\\
    & \leq  \1_{\rho \leq \lambda_u^{\infty}} \left( \alpha\left|\phi(t-u)\right| + \alpha\int_u^t \psi^{(\alpha)}(t-r)\left|\phi(r-u)\right| dr \right) \\
    & \leq \1_{\rho \leq \lambda_u^{\infty}}\psi^{(\alpha)}(t-u).
\end{align*}
\end{proof}

\paragraph{Acknowledgment}: This work received support from the University Research School EUR-MINT
(State support managed by the National Research Agency for Future Investments
 program bearing the reference ANR-18-EURE-0023).

\begin{sloppypar}
\printbibliography
\end{sloppypar}

\end{document}